%% file: revision.tex
\documentclass[10pt,reqno]{amsart} 

\usepackage{float} 

\usepackage{soul}

\usepackage{hhline}

\usepackage{epsfig,amssymb,amsmath,version}
\usepackage{amssymb,version,graphicx,fancybox,mathrsfs}

\usepackage{enumitem}

\usepackage{url,hyperref}
\usepackage{subfigure}
\usepackage{color}
\usepackage{stmaryrd}
\usepackage{multirow}
\usepackage{booktabs,siunitx}
\usepackage{booktabs}
\usepackage{multicol,epstopdf}
\usepackage{xypic}
\usepackage[]{graphicx}
\usepackage[all]{xy}
\usepackage{tikz,ifthen}
\usetikzlibrary{positioning, math, decorations.markings, arrows.meta}
\usetikzlibrary{calc,shapes,cd}
\usetikzlibrary{knots} 
\usepgfmodule{decorations}
\allowdisplaybreaks
\usetikzlibrary{arrows.meta, decorations.markings, positioning}

\tikzset{knotarrow/.pic={ \draw[edge, <-] (0,0) -- +(-.001,0);}}

\tikzset{edge/.style={line width=0.8}}
\tikzset{wall/.style={very thick}}

\tikzset{->-/.style n args={2}{decoration={markings, mark=at position #1 with {\arrow{#2}}}, postaction={decorate}}} 

\ExplSyntaxOn
\cs_new_eq:NN \ifstreqF \str_if_eq:nnF
\cs_new_eq:NN \ifstreqTF \str_if_eq:nnTF
\ExplSyntaxOff

\tikzset{-o-/.code 2 args={\ifstreqF{#2}{} 
{\ifstreqTF{#2}{>}
   {\pgfkeysalso{decoration={markings,mark=at position #1 with {\arrow[scale=0.8]{#2}}}
                    ,postaction={decorate}}
    }
   {\ifstreqTF{#2}{<}
       {\pgfkeysalso{decoration={markings,mark=at position #1 with {\arrow[scale=0.8]{#2}}}
                    ,postaction={decorate}}
        }
       {\pgfkeysalso{decoration={markings,
                    mark=at position #1 with
                    {\draw[black, fill={#2}] circle[radius=2pt];}}
                    ,postaction={decorate}}
        }
     }
  }}}

\allowdisplaybreaks[4]

\newtheorem{theorem}{Theorem}[section]
\newtheorem{lemma}[theorem]{Lemma}
\newtheorem{definition}[theorem]{Definition}
\newtheorem{corollary}[theorem]{Corollary}
\newtheorem{proposition}[theorem]{Proposition}
\newtheorem{remark}[theorem]{Remark}

\newcommand{\bp}{\begin{proposition}}
\newcommand{\ep}{\end{proposition}}
\newcommand{\bpr}{\begin{proof}}
\newcommand{\epr}{\end{proof}}

\newcommand{\bt}{\begin{theorem}}
\newcommand{\et}{\end{theorem}}

\newcommand{\bl}{\begin{lemma}}
\newcommand{\el}{\end{lemma}}

\newcommand{\bcr}{\begin{corollary}}
\newcommand{\ecr}{\end{corollary}}

\newcommand{\be}{\begin{equation}}
\newcommand{\ee}{\end{equation}}

\newcommand{\bes}{\begin{equation*}}
\newcommand{\ees}{\end{equation*}}

\newcommand{\ba}{\begin{align}}
\newcommand{\ea}{\end{align}}

\newcommand{\bas}{\begin{align*}}
\newcommand{\eas}{\end{align*}}

\DeclareMathOperator{\im}{\mathrm{Im}}

\DeclareMathOperator{\Int}{\mathrm{Int}}

\graphicspath{{./Figures/} }

\begin{document}
\bibliographystyle{alpha}

\title{Dehn--Thurston coordinates for ${\rm SL}_3$ basis webs in closed surfaces}

\author[Zhe Sun]{Zhe Sun}
\address{Zhe Sun, School of Mathematical Sciences, University of Science and Technology of China, 96 Jinzhai Road, 230026 Hefei, China}
\email{sunz@ustc.edu.cn}

\author[Zhihao Wang]{Zhihao Wang}
\address{Zhihao Wang, School of Mathematics, Korea Institute for Advanced Study (KIAS), 85 Hoegi-ro, Dongdaemun-gu, Seoul 02455, Republic of Korea}
\email{zhihaowang@kias.re.kr}


 \maketitle

\begin{abstract}
The ${\rm SL}_3$-skein algebra of a closed surface $\Sigma_g$ is a quantization of the
${\rm SL}_3$ character variety of $\Sigma_g$, where $g$ denotes the genus of the surface.
This algebra admits a basis consisting of non-elliptic web diagrams in $\Sigma_g$ after Sikora--Westbury \cite{SW}.
In this paper, we introduce explicit ${\rm SL}_3$ version of Dehn--Thurston coordinates for non-elliptic web diagrams on
$\Sigma_g$, yielding a parametrization by a submonoid of $\mathbb Z^{d}$. The idea is to tropicalize the Fenchel--Nielsen coordinates for the ${\rm SL}_3$-Hitchin component \cite{sun2020flows}.
Here $d = 16g - 16$ for $g \ge 2$ and $d = 4$ in the torus case $g = 1$, coinciding with
the dimension of the corresponding character variety.

\end{abstract}

\tableofcontents

\newcommand{\ca}{{\cev{a}  }}

\def\BZ{\mathbb Z}
\def\Id{\mathrm{Id}}
\def\Mat{\mathrm{Mat}}
\def\BN{\mathbb N}

\def \cb {\color{blue}}
\def \cred {\color{red}}
\def \cbf {\color{blue}\bf}
\def \credf {\color{red}\bf}
\definecolor{ligreen}{rgb}{0.0, 0.3, 0.0}
\def \cg {\color{ligreen}}
\def \cgf {\color{ligreen}\bf}
\definecolor{darkblue}{rgb}{0.0, 0.0, 0.55}
\def \dbf {\color{darkblue}\bf}
\definecolor{anti-flashwhite}{rgb}{0.55, 0.57, 0.68}
\def \afw {\color{anti-flashwhite}}
\def\cF{\mathcal F}
\def\cP{\mathcal P}
\def\embed{\hookrightarrow}
\def\pr{\mathrm{pr}}
\def\cV{\mathcal V}
\def\ot{\otimes}
\def\buu{{\mathbf u}}


\def \ri {{\rm i}}
\newcommand{\bs}[1]{\boldsymbol{#1}}
\newcommand{\cev}[1]{\reflectbox{\ensuremath{\vec{\reflectbox{\ensuremath{#1}}}}}}
\def\cE{\mathcal E}
\def\fB{\mathfrak B}
\def\cR{\mathcal R}
\def\cY{\mathcal Y}
\def\cS{\mathscr S}
\def\cC{\mathcal{C}}

\def\fS{\Sigma}

\def\MN {(M)}
\def\cN {\mathcal{N}}

\def\cP{\mathcal{P}}

\def\bA{\mathbb{A}}

\newcommand{\beq}{\begin{equation}}
	\newcommand{\eeq}{\end{equation}}

\def\SL{{\rm SL}_3}

\section{Introduction}
We will use $\mathbb Z$ and $\mathbb N$ to denote the set of integers and the set of non-negative integers respectively.
Let $R$ be a commutative domain with an invertible element 
$q^{\frac{1}{3}}$.

\subsection{Background}
The $\SL$ webs were introduced by Kuperberg to study the representation theory of $U_q(\mathfrak{sl}_3)$ \cite{kuperberg1996spiders}.
An $\SL$ web is a finite oriented trivalent graph each of whose vertices is either a sink or a source (see Definition~\ref{def.SL3-tangle}).

Let $\fS$ be a marked surface, that is, an oriented surface with punctures and finitely many marked points on the boundary.
The $\SL$-skein algebra $\cS(\fS)$ (see e.g. \cite{SW}) is defined as the quotient $R$-module of the free $R$-module generated by isotopy classes of $\SL$ webs embedded in $\fS \times (-1,1)$, modulo the local relations \eqref{positive-cross}-\eqref{trivial-knot} \cite{kuperberg1996spiders} and some other relations along the boundary \cite{frohman20223}.
The algebraic structure in $\cS(\fS)$ is given by stacking: for two $\SL$ webs $W_1$ and $W_2$ in $\fS \times (-1,1)$, their product $W_1W_2$ is defined by placing $W_1$ above $W_2$.

Actually, the $\SL$-skein algebra is a special case of the ${\rm SL}_n$-skein algebra introduced in \cite{1998Homfly, sikora2001SLn}.
When $n=2$, this construction recovers the well-known Kauffman bracket skein algebra \cite{turaev1988conway,turaev}.
In this case, the set of crossingless multi-curves on a closed surface forms a basis of the ${\rm SL}_2$-skein algebra \cite{przytycki2000skein}, and these basis elements are parametrized by the Dehn--Thurston coordinates (see e.g. \cite{luo2004dehn}). The Dehn--Thurston parameterization plays a crucial role in the study of the center and representation theory of $\operatorname{SL}_2$-skein algebra \cite{frohman2019unicity}, where their seminal work would produce further understanding of the three dimensional quantum invariants in the framework of the representation theory of $\operatorname{SL}_2$-skein algebras.
This motivates us to extend the crucial Dehn--Thurston coordinates from the ${\rm SL}_2$-skein algebra to the $\SL$-skein algebra. 

The $\SL$-skein algebra $\cS(\fS)$ admits a basis $B_\fS$ consisting of all reduced web diagrams in $\fS$ (shown for the polygon case by \cite{kuperberg1996spiders}, the punctured surface case by \cite{SW}, the marked surface case by \cite{frohman20223}), where two web diagrams are identified if they represent isotopic $\SL$ webs in $\fS \times (-1,1)$.
A non-elliptic web diagram is a crossingless web diagram that contains none of the following faces:
\begin{align*}
    \input{2-gon-un},\;\input{4-gon-un},\;
    \input{un-trivial-knot}.
\end{align*}

When each connected component of $\fS$ contains at least one puncture/marked point and $\fS$ admits a triangulation, the basis elements in $B_\fS$ are parametrized by a submonoid of $\BZ^d$ \cite{frohman20223,douglas2024tropical}, where $d$ is the dimension of the corresponding character variety.
Moreover, the coordinates defined in \cite{douglas2024tropical} with respect to different triangulations are shown in \cite{ALCO_2025__8_1_101_0,SSW} that they are related by a sequence of tropical cluster transformations \cite{fockgoncharov06}. 

Such parametrizations have important applications, including the study of the center and representation theory of $\SL$-skein algebras \cite{kim2024unicity}, the construction of quantum trace maps and identification of the highest term exponents of the quantum trace with the coordinates \cite{kim2011sl3}, among others.
We emphasize that there is the cluster ensemble structure on the pair of 
Fock—Goncharov moduli spaces when $\fS$ contains at least one puncture/marked point, but we do not have such a nice structure when $\fS$ is a closed surface. Our new contribution in this paper is the explicit construction of $\SL$ Dehn--Thurston coordinates for the $\SL$-skein algebra, overcoming the noted difficulties.

The crucial idea is that: since the Dehn--Thurston coordinates could be considered as the tropicalization of the Fenchel--Nielsen coordinates of the Teichm\"uller space, our new coordinates should be considered as the tropicalization of the generalized Fenchel--Nielsen coordinates \cite{choi2021symplectic, Kim1999, sun2020flows} for the moduli space of convex $\mathbb{RP}^2$ structures \cite{fock2007moduli, goldman1990convex}.

Let us now describe our new coordinates in more details.
For each positive integer $g$, let $\fS_g$ denote the closed surface of genus $g$.
To construct coordinates for non-elliptic web diagrams in $\fS_g$, we distinguish two cases: $g \ge 2$ and $g = 1$.

\subsection{Main results}
When $g \ge 2$, as in the ${\rm SL}_2$ setting, the construction of coordinates for non-elliptic web diagrams in $\fS_g$ requires a pants decomposition of $\fS_g$.
A \emph{pants decomposition} of $\fS_g$ is a collection of pairwise disjoint, mutually non-parallel, nontrivial simple closed curves $\{C_j\}_{j=1}^{r}$ on $\Sigma_g$, where $r=3g-3$.
We call the decomposition \emph{oriented} if each curve $C_j$ is endowed with an orientation.

For each $j=1,\dots,r$, let $N(C_j)$ be a small closed annular neighborhood of $C_j$ in $\Sigma_g$.
Removing $\bigsqcup_{j=1}^{r} \Int N(C_j)$ from $\Sigma_g$ yields a surface that is a disjoint union of pairs of pants.
We use $\mathbb P$ to denote the set of these pairs of pants.

A \emph{dual graph} is a trivalent graph $\Gamma$ embedded in $\Sigma_g$ such that each curve $C_j$ ($j=1,\dots,r$) intersects $\Gamma$ in exactly one point, and the intersection of $\Gamma$ with each pair of pants consists of a single trivalent vertex.

For each $1 \le j \le r$, let $\mathcal N_j$ denote the surface
\[
\mathcal N_j = N(C_j) \setminus (\partial N(C_j) \cap \Gamma).
\]
To define global coordinates for a non-elliptic web diagram $W$ in $\fS_g$, we first place $W$ in general position with respect to the oriented pants decomposition $\{C_j\}_{j=1}^{r}$ (Definition~\ref{def-general-position}).
We then define local coordinates for each intersection $W \cap \mathcal N_j$, $1 \le j \le r$, and for each intersection $W \cap P$, where $P \in \mathbb P$.

We depict $\mathcal N_j$ as in Figure~\ref{fig:curve-C}, where the red oriented curve represents $C_j$.
An \emph{increasing} (resp.\ \emph{decreasing}) curve is an oriented curve in $\mathcal N_j$ connecting the two boundary components of $\mathcal N_j$ whose orientation is always from left to right (resp.\ from right to left).

\vspace{0.2cm}

To define the local coordinates for $W \cap \mathcal N_j$, we classify the web diagrams $W \cap \mathcal N_j$ into two types (Lemma~\ref{lem-A} and Proposition~\ref{prop-bigon}): strict minimal braided webs (Definition~\ref{def-braid}) and line-circle webs (Figure~\ref{fig:line-circle-web}).
As noted in Definition~\ref{def-braid}, a minimal braided web in $\mathcal N_j$ is obtained from a unique minimal strict braid in $\mathcal N_j$, which consists of a collection of increasing and decreasing curves in $\mathcal N_j$.

When $W \cap \mathcal N_j$ is a strict minimal braided web obtained from a minimal strict braid in $\mathcal N_j$, suppose this minimal strict braid consists of increasing curves
$\{c_1,\dots,c_m\}$ and decreasing curves $\{d_1,\dots,d_l\}$.
We define
\begin{equation}\label{def-intro-twist1}
    \begin{split}
         t_{j1}(W) &:= \text{the sum of the twist numbers of the curves $\{c_1,\dots,c_m\}$},\\
    t_{j2}(W) &:= \text{the sum of the twist numbers of the curves $\{d_1,\dots,d_l\}$},
    \end{split}
\end{equation}
where the twist number of a curve is defined as in Figure~\ref{fig:twist-number}.

When $W \cap \mathcal N_j$ is a line-circle web, we define
\begin{align}\label{def-intro-twist2}
    (t_{j1}(W),\, t_{j2}(W))=
    \begin{cases}
        (t,m), & \text{if $W \cap \mathcal N_j$ is the left configuration in Figure~\ref{fig:line-circle-web}},\\
        (m,t), & \text{if $W \cap \mathcal N_j$ is the right configuration in Figure~\ref{fig:line-circle-web}}.
    \end{cases}
\end{align}

Theorem~\ref{prop-twist} shows that the pair $(t_{j1}(W), t_{j2}(W))$, together with the two intersection numbers (Definition~\ref{def-twist-number}) and the signature of $W \cap \mathcal N_j$ (see \eqref{def-signature}), uniquely determines $W \cap \mathcal N_j$.

\vspace{0.2cm}

We may regard each pair of pants $P \in \mathbb P$ as a thrice-punctured sphere $\fS_{0,3}$, as illustrated in Figure~\ref{fig:sphere}.
Then $W \cap P$ can be viewed as an unbounded lamination in $\fS_{0,3}$ (see \S\ref{sub-unbounded}).
Let $I(\fS_{0,3})$ denote the set of eight vertices in $\fS_{0,3}$ labeled by $11,12,21,22,31,32,v,v'$ as in Figure~\ref{fig:sphere}.

In \cite{ishibashi2025unbounded}, the authors defined the shear coordinates of $W \cap P$ (Lemma~\ref{lem-domain}) 
\[
{\bf x}(W \cap P) = (x_{11}, x_{12}, x_{21}, x_{22}, x_{31}, x_{32}, x_v, x_{v'})
    \in \BZ^{I(\fS_{0,3})}.
\]
We then define
\begin{align*}
    t_P(W) &= x_v - x_{v'},\\
    h_P(W) &= x_{11} - x_{12} + x_{21} - x_{22} + x_{31} - x_{32}.
\end{align*}

Proposition~\ref{prop:intersection-coordinates} and Theorem~\ref{lem:image-P} together imply that the pair $(t_P(W), h_P(W))$, together with the six intersection numbers (Definition~\ref{def:coordinates-pants}), uniquely determines $W \cap P$ up to the moves shown in Figure~\ref{fig:flip-pants}.

\vspace{0.2cm}

Let $C$ be an oriented closed curve in $\fS$.
We say that a web diagram $W$ is in \emph{minimal intersection position} with respect to $C$ if the intersection number between $W$ and $C$ is minimal.

Assume that $W$ is in minimal intersection position with respect to $C$.
By placing $C$ above $W$, we define
\begin{equation*}
\begin{split}
i_1(C,W) &:= \text{the number of positive crossings between $C$ and $W$},\\
i_2(C,W) &:= \text{the number of negative crossings between $C$ and $W$}.
\end{split}
\end{equation*}
Proposition~\ref{prop-bigon} shows that $i_1(C,W)$ and $i_2(C,W)$ are well defined.

For each $1 \le j \le r$, we then define
\begin{align*}
    n_{j1}(W):= i_1(C_j, W), \qquad
    n_{j2}(W):= i_2(C_j, W).
\end{align*}

We now state our first main theorem, which parametrizes the non-elliptic web diagrams in $B_{\fS_g}$ for $g \ge 2$ by a submonoid of $\BZ^{16g-16}$,  where $16g-16$ is the dimension of the corresponding character variety.

\begin{theorem}\label{intro-thm1}
Let $\{C_j\}_{1 \le j \le r}$ be an oriented pants decomposition of $\Sigma_g$ with $g \ge 2$, together with a dual graph $\Gamma$, where $r = 3g-3$.
The coordinate map
\begin{align*}
    \kappa &\colon B_{\fS_g} \longrightarrow
    \BN^r \times \BN^{r} \times \BZ^r \times \BZ^r \times \BZ^{\mathbb P} \times \BZ^{\mathbb P},\\
    W &\mapsto \Bigl(
        (n_{j1}(W))_{1 \le j \le r},\,
        (n_{j2}(W))_{1 \le j \le r},\,
        (t_{j1}(W))_{1 \le j \le r},\,
        (t_{j2}(W))_{1 \le j \le r},\,
        (t_P(W))_{P \in \mathbb P},\,
        (h_P(W))_{P \in \mathbb P}
    \Bigr)
\end{align*}
is injective.
Moreover,
\[
\im \kappa = \Theta,
\]
where $\Theta$ is the submonoid of
$\BN^r \times \BN^{r} \times \BZ^r \times \BZ^r \times \BZ^{\mathbb P} \times \BZ^{\mathbb P}$
defined in Definition~\ref{def-image-theta}.
\end{theorem}

\vspace{0.2cm}

When $g=1$, that is, when the closed surface $\fS_g$ is the torus $\fS_1$, there is no pants decomposition.
Instead, we define coordinates for non-elliptic web diagrams in $\fS_1$ using an oriented closed curve
$\gamma \subset \fS_1$ together with a fixed point $p \in \gamma$.

A non-elliptic web diagram $W$ in $\fS_1$ is said to be in \emph{general position} with respect to $\gamma$
if $W$ is in minimal intersection position with $\gamma$ and satisfies $p \notin W$.

Cutting $\fS_1$ along $\gamma$ and taking the closure yields an annulus, which we denote by $\bA_\gamma$.
The point $p$ gives rise to two points $p', p'' \in \partial \bA_\gamma$, one on each boundary component.
We set
\[
\bA_\gamma' := \bA_\gamma \setminus \{p', p''\}.
\]

Let $W_\gamma$ denote the web diagram in $\bA_\gamma'$ obtained from $W$ by cutting along $\gamma$.

The orientation of $\gamma$ induces orientations on the two boundary components of $\bA_\gamma'$.
There is a unique identification of $\bA_\gamma'$ with the annulus in
Figure~\ref{fig:curve-C} such that the induced orientations
of the boundary components of $\bA_\gamma'$ agree with the orientation of the red curve  in Figure~\ref{fig:curve-C}.
Using this identification, we define the twist coordinates $(t_1(W), t_2(W))$ from $W_\gamma$
as in \eqref{def-intro-twist1} and \eqref{def-intro-twist2}.

We further define the intersection coordinates
\begin{align*}
    n_{1}(W) := i_1(\gamma, W), \qquad
    n_{2}(W) := i_2(\gamma, W).
\end{align*}

We now state our second main theorem, which parametrizes the non-elliptic web diagrams on $\fS_1$
by a submonoid of $\BZ^4$.

\begin{theorem}\label{intro-thm-torus}
Let $\gamma$ be an oriented closed curve in $\fS_1$, and fix a point $p \in \gamma$.
The coordinate map
\begin{align*}
    \kappa \colon B_{\fS_1} \longrightarrow \BN^2 \times \BZ^2, \quad
    W \longmapsto (n_1(W), n_2(W), t_1(W), t_2(W))
\end{align*}
is injective. Moreover,
\[
    \im \kappa
    =
    \Bigl\{
    (n_1,n_2,t_1,t_2) \in \BN^2 \times \BZ^2
    \;\big|\;
    t_i \ge 0 \text{ whenever } n_i = 0,\ i=1,2
    \Bigr\}.
\]
\end{theorem}

The coordinates developed in Theorems~\ref{intro-thm1} and \ref{intro-thm-torus}
have potential applications to the study of the domain property, the Frobenius map \cite{HLW,higgins2025miraculous},
and the center of the $\SL$-skein algebra of closed surfaces, as well as to
$\SL$-Teichm\"uller theory \cite{fockgoncharov06}. We plan to investigate these directions in future work.

{\bf Acknowledgements:}
Z. W. is supported by a KIAS Individual Grant (MG104701) at the Korea Institute for Advanced Study. Z. S. is supported by the NSFC grant 12471068. We thank Tsukasa Ishibashi and Thang T. Q. L{\^e} for helpful discussions.
The authors thank the anonymous referee for helpful comments.

\section{The $\SL$-skein algebra}\label{sec2}

In this section, we review $\SL$-skein algebras in \cite{frohman20223}. 

\subsection{$\SL$-skein algebras}\label{sub111}

\def\cq{\cS}

\def\cM{\mathcal{M}}

A \emph{uni-trivalent graph} $\alpha$ is a finite graph whose vertices all have valency either one or three.
We also allow loop components, that is, connected components without vertices.
An \emph{orientation} of a uni-trivalent graph is an assignment of an orientation to each edge and each loop such that every trivalent vertex is either a \emph{source} or a \emph{sink}.
The \emph{boundary} of $\alpha$, denoted by $\partial\alpha$, is the set of $1$-valent vertices of $\alpha$.
Each point of $\partial\alpha$ is called an \emph{endpoint} of $\alpha$.

\begin{definition}
\label{def:surface}
A {\bf marked surface} $\fS$ is a surface obtained from an oriented compact surface $\overline{\fS}$, possibly with boundary circles, by removing a finite set $\mathcal{M}$ of points.  
The elements of $\mathcal{M}$ are called {\bf marked points} of $\fS$, and those lying in the interior of $\overline{\fS}$ are called the {\bf punctures} of $\fS$, denoted by $\mathring{\mathcal{M}}$.


If $\fS$ has empty boundary, we say that $\fS$ is a {\bf punctured surface}.
A punctured surface is called a {\bf closed surface}
if it contains no punctures.
\end{definition}

\def\wfS{\widetilde{\fS}}

Let $\fS$ be a marked surface and define
\[
\widetilde{\fS} :=\fS \times (-1,1).
\]
We call $\widetilde{\fS}$ the \emph{thickening} of $\fS$, or a \emph{thickened surface}.
For a point $x \in \fS \times (-1,1)$, we refer to its second coordinate (that is, its component in $(-1,1)$) as the \emph{height} of $x$.

\begin{definition}[\cite{frohman20223}]\label{def.SL3-tangle}
    Let $\fS$ be a marked surface.
    A \textbf{web} in $\wfS$ is a properly embedded oriented
    uni-trivalent graph $\alpha$, possibly with oriented loop
    components, satisfying the following conditions:
    \begin{enumerate}[label={\rm (W\arabic*)}]
        \setlength{\itemsep}{0.3em}

        \item Each trivalent vertex is either a source or a sink.

        \item $\alpha$ is equipped with a \emph{framing surface},
        namely a compact oriented ribbon surface
        $F_\alpha\subset\wfS$ having $\alpha$ as a spine.
        The positive normal to $F_\alpha$ determines the
        transversal framing along the edges of $\alpha$.

        \item For each connected component $b$ of $\partial\fS$,
        the endpoints of $\alpha$ in $b\times(-1,1)$
        have mutually distinct heights.

        \item At each endpoint of $\alpha$, the positive normal
        to $F_\alpha$ is parallel to the $(-1,1)$ factor
        and points in its positive direction.
    \end{enumerate}
\end{definition}

We will consider webs in $\wfS$ up to isotopy.
The empty set $\emptyset$ is also considered as a web in $\wfS$. We have the convention that $\emptyset$ is isotopic only to itself. 

Any web $\alpha$ in $\widetilde{\fS}$ can be represented by a {\bf web diagram} in $\fS$. The diagram is just the projection of $\alpha$ on $\fS$. Before projecting, we isotope  $\alpha$ such that the framing is given by the positive direction of $(-1,1)$ and at each singular point there are two transversal strands with over-crossing or under-crossing information. 


Our ground ring is a commutative domain $R$ with an invertible element 
$q^{\frac{1}{3}}$.

The $\SL$-skein module of $\Sigma$, denoted as $\cS(\fS)$, is
the quotient module of the $R$-module freely generated by the set 
of all isotopy classes of 
webs in $\wfS$ subject to  relations \eqref{positive-cross}-\eqref{re-boundary-Y}.

\begin{align}
\label{positive-cross}
    \input{positive-crossing}
    &= q^{-\frac{1}{3}}\,\input{H-re}
       +  q^{\frac{2}{3}}\,\input{smooth-re}, \\
       \label{negative-cross}
    \input{negative-crossing}
    &= q^{\frac{1}{3}}\,\input{H-re}
       +  q^{-\frac{2}{3}}\,\input{smooth-re}, \\
       \label{4-gon}
    \input{4-gon}
    &= \input{smooth2} + \input{smooth1}, \\
    \label{4-gon-re}
    \input{4-gon-rev}
    &= \input{smooth2-re} + \input{smooth1-re}, \\
    \input{2-gon}
    &= -\bigl(q + q^{-1}\bigr)\,\input{straight-line},\\
    \label{trivial-knot}
    \raisebox{-.15in}{\input{trivial-knot1}}&=
     \input{trivial-knot2}= (q^{-2}+1+q^2)
     \raisebox{-.15in}{\begin{tikzpicture}
\filldraw[draw=white,fill=gray!20] (0,0) rectangle (1, 1);
\end{tikzpicture}},\\
\input{trivial-arc} &=\input{trivial-Y}= 0,\label{re-boundary-Y}
\end{align}
where each shaded rectangle in the above relations represents an embedded disk in $\Sigma$. In each relation, the web diagrams coincide outside this disk.

Then  $\cq(\fS)$ admits an algebra structure. For any two  webs $\alpha_1$ and $\alpha_2$ in   $\widetilde{\fS}$, we define $\alpha_1\alpha_2\in \cq(\fS)$ to be the result of stacking $\alpha_1$ above $\alpha_2$. We then refer to $\mathscr{S}(\fS)$ as the $\SL$-{\bf skein algebra} of the surface $\fS$.

\begin{definition}\label{def-resolution}
We call $\input{H-re}$ the \textbf{H-resolution}  of the crossings appearing in \eqref{positive-cross} and \eqref{negative-cross}.  
\end{definition}

\subsection{Bases of the $\SL$-skein algebras}

If two web diagrams in a surface $\fS$ represent two isotopic webs in the thickened surface $\widetilde{\fS}=\fS\times(-1,1)$, we will consider them as the same web diagram. We will implicitly identify a web diagram in $\fS$ with the corresponding isotopy class of webs in $\widetilde{\fS}$.
A web diagram is called \emph{crossingless} if it has no crossings.

\begin{lemma}\cite[Page~6]{frohman20223}\label{lem-flip}
    Any two isotopic crossingless web diagrams differ by a planar isotopy
and flip moves (Figure~\ref{fig:flip}).
\end{lemma}

\begin{figure}
    \centering
    \includegraphics[width=0.5\linewidth]{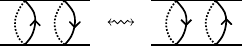}
    \caption{A flip move.}
    \label{fig:flip}
\end{figure}

\begin{definition}\label{def-basis}
    A crossingless web diagram is called {\bf non-elliptic} if 
 it has none of the following elliptic faces:
\begin{align}\label{elliptic-face}
    \input{2-gon-un},\;\input{4-gon-un},\;
    \input{un-trivial-knot},\;
    \input{trivial-arc}, \;
    \input{trivial-Y}
\end{align}
\end{definition}

 By the method of confluence, following~\cite{SW}, we obtain the following result.

\begin{theorem}[\cite{frohman20223}, Proposition~4]\label{def-B_S}
Let $\fS$ be a marked surface.
Then the set of all non-elliptic web diagrams in $\fS$, denoted by $B_{\fS}$, forms an $R$-module basis of the skein algebra $\cS(\fS)$.
\end{theorem}

\begin{remark}
Compared with our definition of the $\SL$-skein algebra, the authors of~\cite{frohman20223} introduce an additional relation, namely relation~(5) in~\cite{frohman20223}, in their definition of the so-called reduced $SU_3$-skein algebra.
Consequently, in addition to the elliptic faces listed in~\eqref{elliptic-face}, their theory excludes one further type of elliptic face, depicted in $\input{H-boundary}$.
As a result, the basis of their reduced $SU_3$-skein algebra consists of crossingless web diagrams without faces of the types appearing in~\eqref{elliptic-face} or in $\input{H-boundary}$; see~\cite[Proposition~4]{frohman20223}.
\end{remark}

\subsection{Graded $\SL$-skein algebras}\label{sub-graded}

Let $C$ be an oriented closed curve in $\fS$, and let $W$ be a crossingless web diagram on $\fS$.
Define $i(C,W)$ to be the minimal geometric intersection number between $C$ and $W$.
We say that $W$ is in \emph{minimal intersection position} with respect to $C$ if this minimum is realized.

Assume that $W$ is in minimal intersection position with respect to $C$.
Placing $C$ above $W$, we define
\begin{equation}\label{def-i+-}
\begin{split}
i_1(C,W) &:= \text{the number of positive crossings between $C$ and $W$},\\
i_2(C,W) &:= \text{the number of negative crossings between $C$ and $W$}.
\end{split}
\end{equation}
We will prove that $i_1(C,W)$ and $i_2(C,W)$ are well-defined
(Proposition~\ref{prop-bigon}).

Let $\cC=\{C_1,\cdots, C_n\}$ be a collection of non-parallel disjoint closed curves in $\fS$. 
Define 
\begin{align*}
    i(\cC,W):= \sum_{1\leq j\leq n} i(C_j,W).
\end{align*}

\def\sp{R\text{-span}}

\def\sk{\cS(\fS)}
\def\gskk{\cS_{\cC}(\fS)}

For each $n\in\mathbb{N}$, define
\begin{align*}
    F_{\le n} := R\text{-span}\{\, W \in B_{\fS} \mid i(\cC,W) \leq n \,\}.
\end{align*}
We define the associated graded algebra
\begin{align}\label{def-graded-algebra}
    \cS_{\cC}(\fS) := \bigoplus_{n\in\mathbb{N}} F_{\le n} / F_{\le n-1},
\end{align}
where we set $F_{\le -1} := 0$.

For any $W \in \sk$ with $W \in F_{\le n} \setminus F_{\le n-1}$, we denote by
\[
    [W] := W + F_{\le n-1} \in \cS_{\cC}(\fS)
\]
its image in the graded algebra.
Theorem~\ref{def-B_S} immediately implies the following result.

\begin{lemma}\label{lem-graded}
The set $\{\, [W] \mid W \in B_{\fS} \,\}$ forms an $R$-basis of $\cS_{\cC}(\fS)$.
\end{lemma}

\section{Unbounded $\SL$-laminations and their shear coordinates}

In this section, we review unbounded $\SL$-laminations on a punctured surface and their shear coordinates, following \cite{ishibashi2025unbounded}. These results will be used in \S\ref{Sec-annulus} and \S\ref{Sec-pants}.
In this section, we will assume that the surface $\Sigma$ is a punctured surface. 

\subsection{Unbounded $\SL$-laminations}\label{sub-unbounded}

An \emph{unbounded $\SL$-lamination} on a punctured surface $\fS$ is an immersed, oriented, uni–trivalent graph $W$ in $\overline{\fS}$ such that each univalent vertex lies in $\mathcal{M}$, while the remaining part of $W$ is embedded in $\fS$.
Moreover, $W$ contains none of the elliptic faces listed in \eqref{elliptic-face}, nor any of the elliptic faces associated with punctures in $\mathcal{M}$, namely
\begin{align}
\input{elliptic-circle}
\end{align}

We consider unbounded $\SL$-laminations up to flip moves (Figure~\ref{fig:flip}) and the following equivalence moves:
\begin{align}\label{eq-flip-lines}
    \raisebox{-.15in}{
	\begin{tikzpicture}
		\tikzset{->-/.style=			
			{decoration={markings,mark=at position #1 with									{\arrow{latex}}},postaction={decorate}}}
        \filldraw[draw=white,fill=gray!20] (0,0) rectangle (2, 1);
        \draw[fill=black,line width =0.8pt] (0.2,0.5) circle (0.05cm);
        \draw[fill=black,line width =0.8pt] (1.8,0.5) circle (0.05cm);
        \draw[line width =0.8pt,decoration={markings, mark=at position 0.7 with {\arrow{<}}},postaction={decorate}]
  plot[smooth] coordinates {(0.2,0.5) (1,0.8) (1.8,0.5)};
  \draw[line width =0.8pt,decoration={markings, mark=at position 0.7 with {\arrow{>}}},postaction={decorate}]
  plot[smooth] coordinates {(0.2,0.5) (1,0.2) (1.8,0.5)};
	\end{tikzpicture}
}\quad
{\Large \sim} \quad
\raisebox{-.15in}{
	\begin{tikzpicture}
		\tikzset{->-/.style=			
			{decoration={markings,mark=at position #1 with										{\arrow{latex}}},postaction={decorate}}}
        \filldraw[draw=white,fill=gray!20] (0,0) rectangle (2, 1);
        \draw[fill=black,line width =0.8pt] (0.2,0.5) circle (0.05cm);
        \draw[fill=black,line width =0.8pt] (1.8,0.5) circle (0.05cm);
        \draw[line width =0.8pt,decoration={markings, mark=at position 0.7 with {\arrow{>}}},postaction={decorate}]
  plot[smooth] coordinates {(0.2,0.5) (1,0.8) (1.8,0.5)};
  \draw[line width =0.8pt,decoration={markings, mark=at position 0.7 with {\arrow{<}}},postaction={decorate}]
  plot[smooth] coordinates {(0.2,0.5) (1,0.2) (1.8,0.5)};
	\end{tikzpicture}
},
\end{align}
\begin{align}\label{eq-155555}
    \raisebox{-.15in}{
	\begin{tikzpicture}
		\tikzset{->-/.style=
						{decoration={markings,mark=at position #1 with
										{\arrow{latex}}},postaction={decorate}}}
        \filldraw[draw=white,fill=gray!20] (-0.25,-0.2) rectangle (1.25, 1);
        \draw[fill=black,line width =0.8pt] (0.5,0) circle (0.05cm);
        \draw[line width =0.8pt] (0.5,0)--(-0.25,0.75);
        \draw[line width =0.8pt] (0.5,0)--(1.25,0.75);
        \draw[line width =0.8pt] (1,0.5)--(0,0.5);
	\end{tikzpicture}
}\quad
{\Large \sim} \quad
\raisebox{-.15in}{
	\begin{tikzpicture}
		\tikzset{->-/.style=
						{decoration={markings,mark=at position #1 with
								{\arrow{latex}}},postaction={decorate}}}
        \filldraw[draw=white,fill=gray!20] (-0.25,-0.2) rectangle (1.25, 1);
        \draw[fill=black,line width =0.8pt] (0.5,0) circle (0.05cm);
        \draw[line width =0.8pt] (0.5,0)--(-0.25,0.75);
        \draw[line width =0.8pt] (0.5,0)--(1.25,0.75);
	\end{tikzpicture}
},
\end{align}
\begin{align}
    \raisebox{-.10in}{
	\begin{tikzpicture}
		\tikzset{->-/.style=
					{decoration={markings,mark=at position #1 with
						{\arrow{latex}}},postaction={decorate}}}
        \filldraw[draw=white,fill=gray!20] (0,0) rectangle (0.8, 0.8);
        \draw[line width =0.8pt] (0.4,0.4) circle (0.3cm);
        \draw[fill=black,line width =0.8pt] (0.4,0.4) circle (0.05cm);
	\end{tikzpicture}
}\quad
{\Large \sim} \quad
\raisebox{-.10in}{
	\begin{tikzpicture}
		\tikzset{->-/.style=
			{decoration={markings,mark=at position #1 with
			{\arrow{latex}}},postaction={decorate}}}
        \filldraw[draw=white,fill=gray!20] (0,0) rectangle (0.8, 0.8);
        \draw[fill=black,line width =0.8pt] (0.4,0.4) circle (0.05cm);
	\end{tikzpicture}
},
\end{align}
\begin{align}\label{eq-circle}
    \begin{array}{c}\includegraphics[scale=0.6]{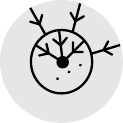}\end{array}
    \quad
{\Large \sim} \quad
\begin{array}{c}\includegraphics[scale=0.6]{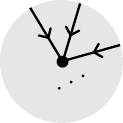}\end{array},
\end{align}
\begin{align}\label{eq-circle1111}
   \text{the move obtained by reversing the orientations of the two diagrams in \eqref{eq-circle}.}
\end{align}
We use $\mathcal{L}(\Sigma)$ to denote the set of unbounded $\SL$-laminations on $\fS$.

\begin{remark}
In \cite{ishibashi2025unbounded}, the authors study signed rational unbounded $\SL$-laminations, where each endpoint of the underlying uni–trivalent graph is equipped with a sign `$+$' or `$-$', and each connected component is labeled by a rational number, modulo the equivalence relations given in \cite[Definition~2.6]{ishibashi2025unbounded}.  

The space $\mathcal{L}(\Sigma)$ considered here corresponds to the subspace consisting of $\SL$-laminations with only `$+$' signs at their endpoints and integral labels. This restricted setting is sufficient for our purposes and aligns naturally with the framework of the present paper; therefore, we adopt $\mathcal{L}(\Sigma)$ rather than the more general construction in \cite{ishibashi2025unbounded}.
\end{remark}

\def\bD{\mathbb{D}}

For each positive integer $k$, we use $\mathbb D_k$ to denote the marked surface obtained from a closed disk by removing $k$ points from its unique boundary component.

To study the unbounded $\SL$-laminations, we first review the unbounded essential webs in $\bD_2$ and $\bD_3$.

Let $E_L$ and $E_R$ denote the boundary intervals of a biangle $\bD_2$. A \emph{finite symmetric strand set} on $\bD_2$ is a pair $S_{\bD_2}=(S_L,S_R)$, where $S_L$ and $S_R$ are finite collections of pairwise disjoint oriented strands (i.e.\ germs of oriented arcs). For each $Z\in\{L,R\}$, the strands in $S_Z$ are located on $E_Z$, and the number of incoming (resp. outgoing) strands on $E_L$ is equal to the number of outgoing (resp. incoming) strands on $E_R$ (see \cite[Figure~5]{ishibashi2025unbounded}).

Given a symmetric strand set $S_{\mathbb D_2}=(S_L,S_R)$, the associated \emph{ladder-web} $W(S)$ on $\bD_2$ is constructed as follows. First, let $W_{\mathrm{br}}(S)$ denote the unique collection of oriented curves connecting the strands in $S_L$ to those in $S_R$ in an order-preserving and minimally intersecting manner. This collection is characterized by a pairing map
\[
f_W \colon S_L \longrightarrow S_R,
\]
which is an order-preserving bijection that sends each incoming (resp. outgoing) strand in $S_L$ to an outgoing (resp. incoming) strand in $S_R$. The ladder-web $W(S)$ is then obtained from $W_{\mathrm{br}}(S)$ by resolving each crossing using the H-resolution (Definition~\ref{def-resolution}).

\begin{definition}\cite[Definition~3.2]{ishibashi2025unbounded}\label{def-bigon}
An \textbf{asymptotically periodic symmetric strand set} $S_{\bD_2}=(S_L,S_R)$ on $\bD_2$ consists of countable collections $S_L$ and $S_R$ of pairwise disjoint oriented strands, where the strands in $S_Z$ are located on $E_Z$ and have no accumulation points, for each $Z\in\{L,R\}$. The oriented strands are required to be symmetric and periodic outside a compact set; see \cite[Figure~6]{ishibashi2025unbounded}.

More precisely, we require that there exists a compact strip $K\subset \bD_2\setminus \mathcal M$ such that:
\begin{itemize}
\item $K$ is bounded by two parallel arcs $\alpha_1$ and $\alpha_2$, transverse to the boundary intervals of $\bD_2$, with $\alpha_1\cup\alpha_2$ disjoint from the strand sets $S_L$ and $S_R$;
\item the pair $(S_L\cap K,\, S_R\cap K)$ is a finite symmetric strand set;
\item the orientation patterns of the strands in $S_L$ and $S_R$ lying in $\bD_2\setminus K$ are periodic, and the pairing map
\[
f_K \colon S_L\cap K \longrightarrow S_R\cap K
\]
of the finite symmetric strand set extends to an order-preserving bijection
\[
f \colon S_L \longrightarrow S_R
\]
that sends each incoming (resp. outgoing) strand of $S_L$ to an outgoing (resp. incoming) strand of $S_R$.
\end{itemize}

Unlike the finite case, the pairing map $f$ need not be unique, as it may depend on the choice of the compact strip $K$. Given such a pair $(S,f)$, one obtains a collection $W_{\mathrm{br}}(S,f)$ of oriented curves in mutual minimal position, and the associated ladder-web $W(S,f)$ is constructed in the same manner as in the finite case. We call $W(S,f)$ the \textbf{ladder-web associated with the pair $(S,f)$}.
We require that there is no crossing between oriented curves in $W_{\mathrm{br}}(S,f)$ whose endpoints lie in $\partial\bD_2\setminus K$.

\end{definition}

\begin{definition}\cite[Definition~3.3]{ishibashi2025unbounded}\label{def-essential-bigon}
An \textbf{unbounded essential web} in $\bD_2$ is an isotopy class of the ladder-web associated with a pair $(S,f)$ as above.
\end{definition}

We now look at webs in $\mathbb D_3$.

The \emph{honeycomb} of degree $d$ (with $d\in\BZ$), denoted by $\mathcal H_d$, is a crossingless web in $\mathbb D_3$.
Figure~\ref{crossbar} illustrates the honeycomb $\mathcal H_{-3}$; reversing the orientation yields $\mathcal H_{3}$.
By convention, $\mathcal H_0$ is the empty web in $\mathbb D_3$.
We define the \emph{weight} of the honeycomb $\mathcal H_d$ to be $|d|$.

\begin{figure}
	\centering
	\includegraphics[width=8cm]{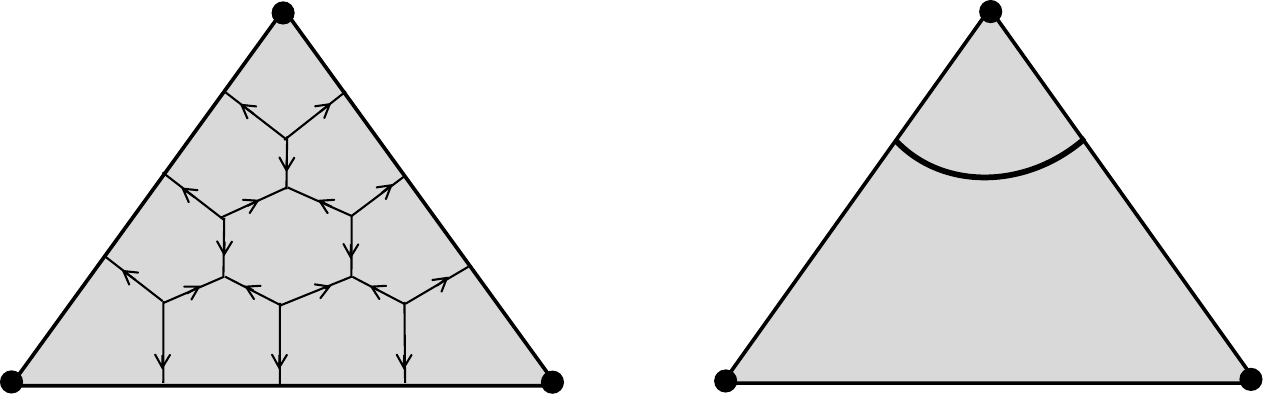}
	\caption{The left picture is an example for $\mathcal H_{-3}$ in $\mathbb D_3$, the right picture is an example for a corner arc in $\mathbb D_3$ (the orientation of the corner arc is arbitrary).}\label{crossbar}
\end{figure}

An essential web in $\bD_3$ consists of a unique (possibly empty) honeycomb component together with a collection of pairwise disjoint oriented arcs located at the corners of $\bD_3$ (see Figure~\ref{crossbar}). These oriented arcs are called corner arcs. As in the biangle case, we define the following unbounded version.

\begin{definition}\cite[Definition~3.6]{ishibashi2025unbounded}
An \textbf{unbounded essential web} in $\bD_3$ is the isotopy class of a disjoint union of a (possibly empty) essential web in $\bD_3$ and, at each corner, at most one semi-infinite periodic collection of corner arcs.
\end{definition}

To introduce the $\SL$-shear coordinates, we associate to each unbounded $\SL$-lamination a spiralling diagram as follows.

\begin{definition}\label{def-spiralling}
Let $W$ be an unbounded $\SL$-lamination on $\Sigma$. The associated \emph{spiralling diagram} $\widetilde W$ is constructed by the following procedure. For each puncture $p\in\mathcal{M}$, choose a sufficiently small disk neighborhood $D_p$ of $p$ and deform each end of $W$ incident to $p$ into an infinitely spiralling curve, spiralling in the clockwise direction, as illustrated in Figure~\ref{fig:spiralling}.
\end{definition}

\begin{figure}
    \centering
    \includegraphics[width=0.25\linewidth]{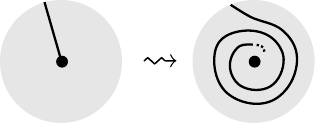}
    \caption{Clockwise direction of spiralling.}
    \label{fig:spiralling}
\end{figure}

We say that $\fS$ is 
{\bf triangulable} if each component of $\fS$ is not 
a sphere with 
less than three punctures.

    An embedding $c:(0,1)\rightarrow \fS$ is called an {\bf ideal arc} if both
$\bar c(0)$ and $\bar c(1)$ are punctures, where $\bar c\colon [0,1] \to \overline \fS$ is the `closure' of $c$. By an ideal arc we often mean its image in $\fS$. 
An {\bf ideal triangulation}, or a {\bf triangulation} $\Delta$ of $\fS$, is a collection of mutually disjoint ideal arcs in $\fS$ with the following properties:
\begin{enumerate}[label={\rm (T\arabic*)}]
    \item any two arcs in $\Delta$ are not isotopic;

    \item $\Delta$ is maximal under condition (T1);

    \item 
    the valence of $\Delta$ at each puncture is at least two.
\end{enumerate}

For each edge $E$ of $\Delta$, we take two disjoint copies $E'$ and $E''$ of $E$ such that no two edges in the collection
\[
\bigcup_{E\in\Delta}\{E',E''\}
\]
intersect. We define
\begin{align}
\label{widehat_lambda}
\widehat{\Delta} := \bigcup_{E\in\Delta}\{E',E''\},
\end{align}
and call $\widehat{\Delta}$ a \textbf{split} triangulation of $\Delta$.
Cutting the surface $\fS$ along the edges of $\widehat{\Delta}$ yields a collection of bigons and triangles. For each $E\in e(\Delta)$, we denote by $B_E$ the bigon in $\widehat{\Delta}$ bounded by $E'$ and $E''$, and we denote by $t(\Delta)$ the set of triangular faces of $\widehat{\Delta}$.

\begin{definition}
The spiralling diagram $\widetilde W$ is said to be \textbf{in good position} with respect to a split triangulation $\widehat{\Delta}$ if, for each $E\in e(\Delta)$ and each $T\in t(\Delta)$, both intersections $\widetilde W\cap B_E$ and $\widetilde W\cap T$ are unbounded essential webs.
\end{definition}

Any spiralling diagram arising from an unbounded lamination on $\Sigma$ can be isotoped into good position with respect to $\widehat{\Delta}$ \cite[Theorem~3.10]{ishibashi2025unbounded}.

\subsection{$\SL$-shear coordinates}

\def\LfS{\mathcal{L}(\Sigma)}

Before introducing the $\SL$–shear coordinates, we first fix some equivalent notations for honeycombs.
The following three diagrams all represent a honeycomb of weight $n$.
Its degree is $n$ (resp., $-n$) if the strands incident to the triangle are oriented towards (resp. away from) the triangle:
\begin{align}\label{eq-honeycomb-n}
\begin{array}{c}
    \includegraphics[scale=0.8]{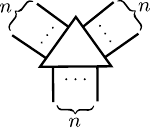}
\end{array}
\;\sim\;
\begin{array}{c}
    \includegraphics[scale=0.8]{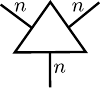}
\end{array}
\;\sim\;
\begin{array}{c}
    \includegraphics[scale=0.8]{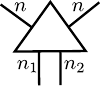}
\end{array},\qquad \text{where $n=n_1+n_2$}.
\end{align}

We equip $\Delta$ with two distinguished points in the interior of each edge and one distinguished point in the interior of each triangle. We denote the set of all such points by
$I(\Delta)$.

Let $W \in \mathcal{L}(\Sigma)$ be an $\SL$-lamination.  Let $\widetilde W$ denote the associated spiralling diagram,  placed in good position with respect to $\widehat{\Delta}$. Let $W^{\Delta}_{\mathrm{br}}$ be its braid representative. The shear coordinates of $W$ will be defined in terms of $W^{\Delta}_{\mathrm{br}}$.

For each $E\in \Delta$, let $Q_E$ be the unique quadrilateral having $E$ as its diagonal, viewed as the union of two triangles $T_L$ and $T_R$ together with the biangle $B_E$. The restriction of $W^{\Delta}_{\mathrm{br}}$ to each of $T_L$ and $T_R$ contains at most one honeycomb web, represented by a triangular symbol as in \eqref{eq-honeycomb-n}. Any strand of the braid representative $W^{\Delta}_{\mathrm{br}}\cap Q_E$ that is incident to the triangular symbol in $T_L$ (if it exists) is called a \emph{$T_L$-strand}. Similarly, one defines \emph{$T_R$-strands}. It may happen that a strand is both a $T_L$- and a $T_R$-strand, in which case it connects the two honeycombs. After removing all $T_L$- and $T_R$-strands, the remaining diagram consists of a collection of (possibly intersecting) oriented curves, which we call the \emph{curve components}; see Figure~\ref{fig:curve-components}.

\begin{definition}\cite[Definition~3.12]{ishibashi2025unbounded}
The $\SL$-shear coordinate system
\[
\mathbf{x}(W)
=\bigl(x_{i}(W))_{i\in I(\Delta)}
\in \mathbb{Z}^{I(\Delta)}
\]
is defined as follows. For each  edge $E\in \Delta$, the coordinates assigned to the four points in the interior of the quadrilateral $Q_E$ depend only on the restriction $W^{\Delta}_{\mathrm{br}}\cap Q_E$.

\begin{enumerate}
\item Each curve component contributes to the edge coordinates according to the rule illustrated in \eqref{fig:curve-components}.
\item The honeycomb in the triangle $T_L$ contributes to $\mathbf{x}(W)$ as shown in \eqref{fig:honeycombs}. 
\item The honeycomb in the triangle $T_R$ and the $T_R$-strands contribute symmetrically, via a $\pi$-rotation of the configuration in \eqref{fig:honeycombs}.
\end{enumerate}
\end{definition}

\begin{equation}\label{fig:curve-components}
    \begin{split}
        \begin{array}{c}\includegraphics[scale=0.8]{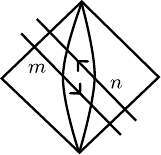}\end{array}
    \overset{{\bf x}}{\mapsto}
\begin{array}{c}\includegraphics[scale=0.8]{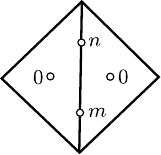}\end{array}\\
\begin{array}{c}\includegraphics[scale=0.8]{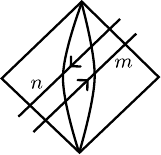}\end{array}
    \overset{{\bf x}}{\mapsto}
\begin{array}{c}\includegraphics[scale=0.8]{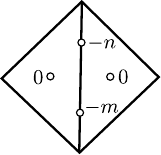}\end{array}
    \end{split}
\end{equation}

\begin{equation} \label{fig:honeycombs}
    \begin{split}
        \begin{array}{c}\includegraphics[scale=0.8]{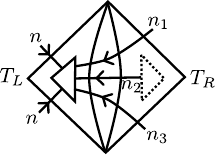}\end{array}
    \overset{{\bf x}}{\mapsto}
\begin{array}{c}\includegraphics[scale=0.8]{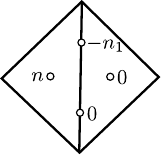}\end{array}\\
\begin{array}{c}\includegraphics[scale=0.8]{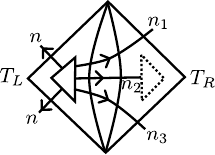}\end{array}
    \overset{{\bf x}}{\mapsto}
\begin{array}{c}\includegraphics[scale=0.8]{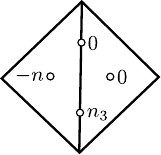}\end{array}
    \end{split}
\end{equation}

\begin{remark}
The Fock--Goncharov $p$ map \cite{fockgoncharov06} $\mathcal{A}\rightarrow \mathcal{X}$ relates the tropical $\mathcal{A}$ coordinates in \cite{douglas2024tropical} and the $\SL$-shear coordinates by $p^*(x_i)=\sum_j \varepsilon_{ij} a_j$. One can also derive the $\SL$-shear coordinates by $p$ map and the tropical $\mathcal{A}$ coordinates from the square case, which is shown in the computation right after \cite[Definition 4.13]{ALCO_2025__8_1_101_0} or \cite[arXiv version 1, Remark 6.4]{ALCO_2025__8_1_101_0}. 
\end{remark}

It is well-known that \cite[Theorem~3.19]{ishibashi2025unbounded}
$$\mathbf{x}\colon \LfS\rightarrow \mathbb{Z}^{I(\Delta)},\qquad 
W\mapsto \bigl(x_{i}(W))_{i\in I(\Delta)}$$
is a well-defined injection. 

Let $\Sigma_{0,3}$ denote the thrice–punctured sphere. Let $\lambda$ be the unique ideal triangulation of $\Sigma_{0,3}$, and label the set $I(\lambda)$ as in Figure~\ref{fig:sphere}. 
Since the triangulation of $\Sigma_{0,3}$ is unique, we denote 
$I(\lambda)$ by $I(\Sigma_{0,3})$. 
Define
\begin{align}\label{def-lambda}
\Lambda
:=
\Bigl\{
&x_{11}+x_{32}\ge 0,\quad
x_{12}+x_{31}+x_v+x_{v'}\ge 0,\quad
x_{31}+x_{22}\ge 0, \nonumber\\
&x_{32}+x_{21}+x_v+x_{v'}\ge 0,\quad
x_{21}+x_{12}\ge 0,\quad
x_{22}+x_{11}+x_v+x_{v'}\ge 0
\Bigr\} \subset \mathbb{Z}^{I(\Sigma_{0,3})}.
\end{align}

\begin{figure}
    \centering
    \includegraphics[width=0.25\linewidth]{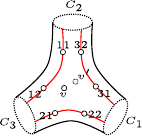}
    \caption{We depict $\Sigma_{0,3}$ as a pair of pants with its boundary removed. The triangulation is given by the three red curves. The set $I(\Sigma_{0,3})$ is labeled by $11,12,21,22,31,32,v,v'$, where $v$ (resp. $v'$) lies in the front (resp. back) triangle. The three punctures of $\Sigma_{0,3}$ (equivalently, the three boundary components of the pair of pants) are labeled by $C_1,C_2,C_3$.}
    \label{fig:sphere}
\end{figure}

We thank Tsukasa Ishibashi for pointing out the following result, which follows immediately from \cite[Theorem~3.19]{ishibashi2025unbounded} and \cite[Proposition~4.4]{ishibashi2024unbounded}.

\begin{lemma}\label{lem-domain}
The coordinate map
\[
\mathbf{x}\colon \mathcal{L}(\Sigma_{0,3}) \longrightarrow \mathbb{Z}^{I(\Sigma_{0,3})}
\]
is injective, and its image is precisely $\Lambda$.
\end{lemma}

\section{General positions of basis elements in pants decompositions of a closed surface}

A pair of pants is a surface homeomorphic to $S^2$ with three open disks removed.

We use $P$ to denote a pair of pants, whose three boundary components are labeled by
$C_1, C_2,$ and $C_3$.
Let $P'$ be the marked surface obtained from $P$ by removing one point from each boundary component.

\begin{definition}\label{def-essential-pants}
Let $\gamma_1,\gamma_2,\gamma_3$ be the seams (the blue arcs in
Figures~\ref{fig:W-to-W1} and \ref{fig:eneral-position}), and set $\gamma:=\gamma_1\cup\gamma_2\cup\gamma_3$.
A non-elliptic web diagram in $P'$ is called \textbf{essential} if its
isotopy class admits a fixed embedded representative $W$ such that:
\begin{enumerate}[label={\rm (C\arabic*)}]
    \item Let $W_1$ be obtained from $W$ by a full twist along each of
    $C_1,C_2,C_3$ in the direction of Figure~\ref{fig:W-to-W1}. Then
    $W_1$ is in minimal intersection position with
    $\gamma_1,\gamma_2,\gamma_3$.
    \item Let $C_1',C_2',C_3'$ be any pairwise disjoint curves in $P'$
    parallel to $C_1,C_2,C_3$, each meeting $\gamma$ transversely in
    exactly two points, such that
    \[
        \#(C_j\cap W)=\#(C_j'\cap W)=i(C_j,W),\qquad 1\le j\le 3.
    \]
    Let $A_j$ be the annulus between $C_j$ and $C_j'$. Then $W\cap A_j$
    consists of $i(C_j,W)$ disjoint arcs from $C_j$ to $C_j'$, each
    transverse to the circles parallel to $C_j$, and each meeting $\gamma$
    at most once, in the direction described below.
\end{enumerate}
Here $\#$ denotes the actual number of intersection points. Thus (C2)
means that $W\cap A_j$ agrees with Figure~\ref{fig:eneral-position} up
to an isotopy of $A_j$ preserving $C_j$, $C_j'$ and $\gamma\cap A_j$.
An arc disjoint from $\gamma$ is a straight strand, and an arc meeting
$\gamma$ once turns around $A_j$ clockwise as it approaches $C_j$, with
respect to the orientation of $P$ (regarding $C_j$ as a puncture, this
is the spiralling direction of Figure~\ref{fig:spiralling}).
\end{definition}

\begin{figure}
    \centering
    \includegraphics[width=0.48\linewidth]{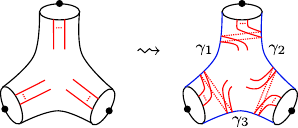}
    \caption{The procedure to obtain $W_1$ from $W$.}
    \label{fig:W-to-W1}
\end{figure}

\begin{figure}
    \centering
    \includegraphics[width=0.21\linewidth]{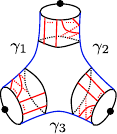}
    \caption{The picture for $W \cap (A_1\cup A_2\cup A_3)$.}
    \label{fig:eneral-position}
\end{figure}

For example, the web diagram in Figure~\ref{fig:good0}(B) does not satisfy {\rm (C1)}, whereas the web diagram in Figure~\ref{fig:good0}(A) is essential. 
The web diagram in Figure~\ref{fig:good0}(D) satisfies neither {\rm (C1)} nor {\rm (C2)}, whereas the web diagram in Figure~\ref{fig:good0}(C) is essential.
The web diagram in Figure~\ref{fig:good0}(F) does not satisfy {\rm (C2)}, whereas the one shown in Figure~\ref{fig:good0}(E) is essential.
One can isotope the web diagram in Figure~\ref{fig:good0}(B) so that it satisfies {\rm (C1)}; however, the resulting diagram fails to satisfy {\rm (C2)}.
It is impossible to isotope the web diagrams in Figure~\ref{fig:good0}(D) and (F) so that they satisfy {\rm (C2)}.

\begin{figure}
    \centering
    \includegraphics[width=0.9\linewidth]{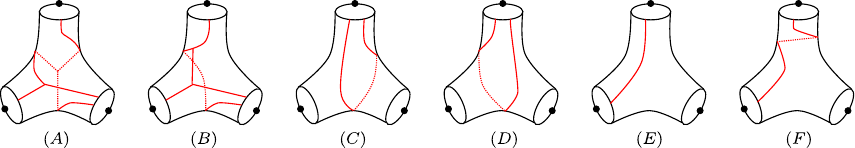}
    \caption{The picture for $W \cap P'$.}
    \label{fig:good0}
\end{figure}

Let $\Sigma_g$ denote a closed surface of genus $g$.
Throughout this section, we assume that $g\geq 2$.
A \emph{pants decomposition} of $\Sigma_g$ is a collection
$\{C_i\}_{i=1}^{3g-3}$ of pairwise disjoint, mutually
non-parallel, nontrivial simple closed curves on $\Sigma_g$.
We call $\{C_i\}_{i=1}^{3g-3}$ \emph{oriented} if each $C_i$ is oriented.
For each $i=1,\dots,3g-3$, let $N(C_i)$ be a small closed annular neighborhood of $C_i$ in $\Sigma_g$.
Removing $\bigsqcup_{i=1}^{3g-3} \Int N(C_i)$ from $\Sigma_g$ yields a surface that is a disjoint union of pairs of pants.

A \emph{dual graph} is a trivalent graph $\Gamma$ embedded in $\Sigma_g$ such that each curve $C_i$ ($i=1,\dots,3g-3$) intersects $\Gamma$ in exactly one point, and the intersection of $\Gamma$ with each pair of pants consists of a single trivalent vertex.

Let $P$ be a pair of pants associated with the pants decomposition
$\{C_i\}_{i=1}^{3g-3}$.
We say that $P$ is \emph{bounded by $C_{i_1},C_{i_2},C_{i_3}$}
if its three boundary components can be labeled
$b_1,b_2,b_3$ so that $b_j$ is a boundary component of
$N(C_{i_j})$ for each $j=1,2,3$.
The indices $i_1,i_2,i_3$ need not be distinct:
both boundary components of the same annulus $N(C_i)$
may occur as boundary components of $P$.
By abuse of notation, we refer to $b_j$ simply as $C_{i_j}$.

\begin{definition}\label{def-general-position}
A non-elliptic web diagram $W$ in $\Sigma_g$ is said to be in \textbf{general position} with respect to the pants decomposition
$\{C_i\}_{i=1}^{3g-3}$ if, for each $1 \le i \le 3g-3$, the following conditions are satisfied:
\begin{enumerate}[label={\rm (\alph*)}, itemsep=0.3em]
    \item $W$ is in minimal position with respect to $\partial N(C_i)$;
    \item $W \cap \partial N(C_i) \subset \partial N(C_i) \setminus \Gamma$;
    \item Let $P$ be a pair of pants associated to $\{C_i\}_{i=1}^{3g-3}$, and let $P'=P\setminus (\partial P\cap \Gamma)$. Then $W\cap P'$ is an essential non-elliptic web diagram in $P'$ (Definition~\ref{def-essential-pants}).
\end{enumerate}
\end{definition}

When the pants decomposition is clear from the context, we will omit the phrase
``with respect to the pants decomposition $\{C_i\}_{i=1}^{3g-3}$''.

Any non-elliptic web diagram $W$ can be isotoped into general
position, as follows.

\smallskip\noindent\emph{Step 1: conditions {\rm (a)} and {\rm (b)}.}
By \cite[Proposition~11 and Corollary~12]{frohman20223}, $W$ can be
brought into minimal intersection position with all curves
$\partial N(C_i)$ simultaneously by cap, vertex, and crossbar moves
\cite[Figure~9]{frohman20223};
none of these moves increases the intersection number with any of the
curves. Although \cite{frohman20223} treats ideal arcs, the argument
applies to simple closed curves, replacing the bigon in the proof of
\cite[Proposition~11]{frohman20223} by the bigon or annulus cobounded by
two isotopic closed curves and using \cite[Lemmas~14 and~15]{frohman20223}.
A small isotopy then moves $W\cap\partial N(C_i)$ off $\Gamma$.

\smallskip\noindent\emph{Step 2: condition {\rm (c)}.}
Let $P$ be a pair of pants associated to $\{C_i\}_{i=1}^{3g-3}$, and let
$P'=P\setminus (\partial P\cap \Gamma)$. Regard the boundary components
of $P$ as punctures. Using the moves
\eqref{eq-155555}--\eqref{eq-circle1111}, we remove from $W\cap P$
all configurations appearing on the left-hand sides of these moves; each
such move is realised by an isotopy of $W$ in $\Sigma_g$ pushing a
crossbar or a loop across $\partial P$ into the adjacent annulus, and it
does not change the number of intersection points with any
$\partial N(C_i)$. By~(a), $W\cap P$ then defines an unbounded
$\SL$-lamination $L_P\in\mathcal L(\Sigma_{0,3})$ (\S\ref{sub-unbounded}).
We regard the seams $\gamma_1,\gamma_2,\gamma_3$ as ideal arcs; they form
the ideal triangulation of $\Sigma_{0,3}$.
By \cite[Theorem~3.10]{ishibashi2025unbounded}, the spiralling diagram of
$L_P$ (Definition~\ref{def-spiralling}) can be put in good position with
respect to the corresponding split triangulation. Truncating its
spiralling ends near $\partial P$ gives a diagram $W_P$ in $P'$.
Near each $C_j$ the truncated spirals are disjoint arcs turning in the
direction of Figure~\ref{fig:spiralling}, so $W_P$ satisfies (C2).
Performing a full twist along each $C_j$ as in Figure~\ref{fig:W-to-W1}
only extends each truncated spiral by one more turn, so the resulting
diagram is again a truncation of the same spiralling diagram in good
position; hence it meets each seam $\gamma_k$ minimally, which gives (C1).
Since $W\cap P$ and $W_P$ represent the same lamination $L_P$, they differ
by twists along $\partial P$ and the move \eqref{eq-flip-lines}. These are
realised by an isotopy of $W$ in $\Sigma_g$ sliding strands and
crossbars across $\partial P$ into the adjacent annuli (see
Figure~\ref{fig:crross-bar-pass}), which preserves (a) and (b) and
gives $W\cap P'=W_P$. The isotopies for different pairs of pants have
disjoint supports, which gives~(c).

In \S\ref{Sec-annulus} and \ref{Sec-pants}, we analyze the structure of $W$ within each annulus and each pair of pants.

\section{Webs in an annulus}\label{Sec-annulus}

In this section, we classify non-elliptic web diagrams in the annulus (Lemma~\ref{lem-A}).
More precisely, we show that there are exactly two types of non-elliptic web diagrams in the annulus.
One type is obtained by applying the H-resolution to a strict minimal braid in the annulus (Definition~\ref{def-braid}),
while the other type is illustrated in Figure~\ref{fig:line-circle-web}.

Based on this classification, we define two twist numbers, denoted by $t_1$ and $t_2$, for non-elliptic web diagrams in the annulus (Definition~\ref{def-twist-number}).
The main result of this section is Theorem~\ref{prop-twist}, which states that the coordinates defined in Definition~\ref{def-twist-number}
uniquely determine a non-elliptic web diagram in the annulus up to its signature (see \eqref{def-signature}).

The twist numbers introduced here will later be used to define twist numbers for non-elliptic web diagrams in $\Sigma_g$ $(g\geq 2)$,
within each annulus associated to a pants decomposition of $\Sigma_g$.

\subsection{Braids and braided webs in the annulus}
Let $\bA$ denote the annulus $S^1\times[0,1]$. We define 
\begin{align}\label{A01}
    A_0:= S^1\times\{0\},\qquad
    A_1:= S^1\times\{1\}.
\end{align}
For any point $(x,t)\in \bA$, we call $t$ the \emph{height} of the point.
We say that a point $(x,t)$ is \emph{higher} than another point $(x',t')$ if $t>t'$.

When drawing the annulus, we adopt the convention that the height increases from left to right; that is, the left (resp.\ right) boundary component is $A_0$ (resp.\ $A_1$).

An \emph{increasing curve} (resp. a \emph{decreasing curve}) in $\bA$ is a properly embedded oriented map
$c\colon[0,1]\to\bA$ such that
\[
c(t)\in S^1\times\{t\}
\quad
\text{(resp.) } c(t)\in S^1\times\{1-t\}\text{)}
\]
for all $t\in[0,1]$.

\begin{definition}\label{def-braid}

A \textbf{braid} $\mathcal{B}$ in $\bA$ is a finite collection
$\{c_1,\ldots,c_n\}$ of increasing curves, decreasing curves,
and oriented simple closed curves, where each closed curve
lies in the interior of $\bA$ and is not null-homotopic in $\bA$,
such that:
\begin{enumerate}[label={\rm (\alph*)}, itemsep=0.3em]
\item for $i\neq j$, the intersection $c_i\cap c_j$ is a finite set contained in the interior of $\bA$;
\item any two increasing curves (resp. any two decreasing curves) are disjoint; and
\item each closed curve is in minimal intersection position with any $c_i$.
\end{enumerate}
We say that $\mathcal{B}$ is \textbf{minimal} if the curves $c_1,\ldots,c_n$ are in minimal intersection position, and \textbf{strict} if it contains no closed curves.

We allow $\mathcal B$ to be empty. The empty braid is minimal and strict. 
\end{definition}

The crossingless web diagram in $\bA$ obtained from a braid $\mathcal{B}$ by resolving each crossing using the H-resolution (Definition~\ref{def-resolution}) is called a \emph{braided web} in $\bA$, denoted by $H(\mathcal B)$. 
We call $\mathcal B$ a \emph{braid representative} of $H(\mathcal B)$.
Note that a braided web may have different braid representatives. 
Examples are provided by the planar isotopies in
Lemma~\ref{lem-isotopy}: the braids underlying the two
sides, as described in
\eqref{eqqqqq25}--\eqref{eqqqqq28}, are distinct, but their
H-resolutions represent the same braided web.

A braided web in $\bA$ is said to be \emph{minimal} (resp.\ \emph{strict}) if it has a minimal (resp.\ strict) braid representative.


Let $\mathbb A'$ be a marked surface obtained from $\bA$ by removing one point from each boundary component.
Then the above definitions also apply to $\mathbb A'$.
When depicting $\mathbb A'$, we will, for simplicity, omit the two marked points shown in Figure~\ref{fig:Aprime}.

\begin{figure}
    \centering
    \includegraphics[width=0.2\linewidth]{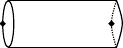}
    \caption{An illustration of $\mathbb A'$.}
    \label{fig:Aprime}
\end{figure}

\def\cB{\mathcal{B}}

Note that a braided web in $\bA'$ (or $\bA$) need not be non-elliptic.
The following lemma provides an equivalent characterization for non-ellipticity, together with a necessary condition.

\begin{lemma}\label{lem-minimal-non}
    Let $\cB=\{c_1,\ldots,c_n\}$ be a braid in $\bA'$. Then:
    \begin{enumerate}[label={\rm (\alph*)}, itemsep=0.3em]
        \item $H(\cB)$ is non-elliptic if and only if
        $H(\cB)$ contains no $4$-gons of the type shown in~\eqref{elliptic-face}.
        \item If $H(\cB)$ is non-elliptic, then $\cB$ is minimal.
    \end{enumerate}
\end{lemma}

\begin{proof}
(a)
The $\Rightarrow$ direction follows directly from Definition~\ref{def-basis}.
For the $\Leftarrow$ direction, observe that among the elliptic faces listed in~\eqref{elliptic-face}, the only ones that can occur in $H(\cB)$ are $4$-gons. 
Hence the absence of such $4$-gons implies that $H(\cB)$ is non-elliptic.

\medskip
(b)
Suppose, for contradiction, that $\cB$ is not minimal.
Then $\cB$ contains a bigon.
By Definition~\ref{def-braid}(b) and (c), this bigon $\mathbb D$ is formed by one decreasing curve and one increasing curve.
We may assume that $\mathbb D$ contains no smaller bigons in its interior.
Let $b_1$ and $b_2$ be the two crossings that bound the bigon $\mathbb D$.

First, assume that no other curves pass through $\mathbb D$.
Then the H-resolutions of $b_1$ and $b_2$ produce a $4$-gon in $H(\cB)$, contradicting the assumption that $H(\cB)$ contains no $4$-gons.

Next, suppose that there are curves $c_{i_1},\ldots,c_{i_m}$ passing through $\mathbb D$.
If some $c_{i_j}$ is an increasing or decreasing curve, then by Definition~\ref{def-braid}(b) it intersects only one boundary edge of $\mathbb D$, which would create a smaller bigon inside $\mathbb D$, contradicting the minimality of $\mathbb D$.
Therefore, all curves $c_{i_1},\ldots,c_{i_m}$ must be closed curves.
By Definition~\ref{def-braid}(c), these closed curves are pairwise disjoint.

If all curves $c_{i_1},\ldots,c_{i_m}$ have the same orientation, let $c_{i_j}$ (resp.\ $c_{i_k}$) be the curve closest to the crossing $b_1$ (resp.\ $b_2$).
By Definition~\ref{def-braid}(c), the curve $c_{i_j}$ (resp.\ $c_{i_k}$) intersects $\mathbb D$ in exactly two points, say $b_3,b_4$ (resp.\ $b_5,b_6$).
Resolving the crossings $b_i$ for $1\le i\le 6$ via the H-resolution produces a $4$-gon in $H(\cB)$, again a contradiction.

Now assume that the curves $c_{i_1},\ldots,c_{i_m}$ do not all have the same orientation.
Order them by height, so that $c_{i_j}$ lies above $c_{i_k}$ whenever $j>k$.
If there exists $2\le j\le m-1$ such that $c_{i_{j-1}}$ and $c_{i_{j+1}}$ have the same orientation, while $c_{i_j}$ has the opposite orientation, then by Definition~\ref{def-braid}(c) the curves
$c_{i_{j-1}}$, $c_{i_j}$, and $c_{i_{j+1}}$ intersect $\mathbb D$ in exactly two points each, say
$d_1,d_2$, $d_3,d_4$, and $d_5,d_6$, respectively.
Resolving the crossings $d_i$ for $1\le i\le 6$ via the H-resolution again produces a $4$-gon in $H(\cB)$, a contradiction.

Otherwise, there exists $1\le k\le m-1$ such that $c_{i_1},\ldots,c_{i_k}$ have the same orientation, while
$c_{i_{k+1}},\ldots,c_{i_m}$ all have the opposite orientation.
By Definition~\ref{def-braid}(c), the curves
$c_{i_1}$, $c_{i_k}$, $c_{i_{k+1}}$, and $c_{i_m}$
each intersect $\mathbb D$ in exactly two points, say
$e_1,e_2$, $e_3,e_4$, $e_5,e_6$, and $e_7,e_8$, respectively.
Resolving the crossings $e_i$, for $1\le i\le 8$, $b_1$, and $b_2$ via the H-resolution again yields a $4$-gon in $H(\cB)$, a contradiction.

These contradictions show that $\cB$ must be minimal.
\end{proof}

\subsection{Non-elliptic web diagrams in the annulus}

In this subsection, we will classify 
non-elliptic web diagrams in $\bA$ (Lemma~\ref{lem-A} and Proposition~\ref{prop-bigon}). We first introduce some notations. 

Let $c$ be a decreasing or increasing curve in $\mathbb A'$. We define its \emph{twist number} $t(c)$ as illustrated in Figure~\ref{fig:pos-neg}.

\begin{figure}
    \centering
    \includegraphics[width=0.5\linewidth]{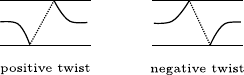}
    \caption{The twist number of the left (resp.\ right) picture is $1$ (resp.\ $-1$).}
    \label{fig:pos-neg}
\end{figure}

For $n\in \mathbb N$ and $t\in\mathbb Z$ ($t=0$ if $n=0$),
we use the symbol
\[
\begin{array}{c}\includegraphics[scale=1.5]{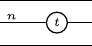}\end{array}
\]
to denote the diagram obtained from $n$ parallel strands by introducing twists so that the total twist number is $t$; see Figure~\ref{fig:twist-number} for one example.

\begin{figure}[H]
    \centering
    \includegraphics[width=0.5\linewidth]{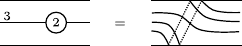}
    \caption{One example of twisted parallel strands.}
    \label{fig:twist-number}
\end{figure}

For $m,n\in \mathbb N$,
we introduce the following notations:
\begin{align}
\label{eqqqqq25}
    \begin{array}{c}\includegraphics[scale=1.5]{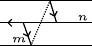}\end{array}
    &:= \text{the H-resolution of }
    \begin{array}{c}\includegraphics[scale=1.5]{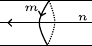}\end{array},\\
    \begin{array}{c}\includegraphics[scale=1.5]{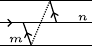}\end{array}
    &:= \text{the H-resolution of }
    \begin{array}{c}\includegraphics[scale=1.5]{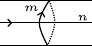}\end{array},\\
    \begin{array}{c}\includegraphics[scale=1.5]{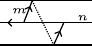}\end{array}
    &:= \text{the H-resolution of }
    \begin{array}{c}\includegraphics[scale=1.5]{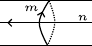}\end{array},\\
    \label{eqqqqq28}
    \begin{array}{c}\includegraphics[scale=1.5]{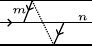}\end{array}
    &:= \text{the H-resolution of }
    \begin{array}{c}\includegraphics[scale=1.5]{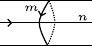}\end{array},
\end{align}
where on right hand side of `$:=$' the integer $n$ or $m$ represents the number of copies of the corresponding curves. 

The following describes certain planar isotopies for web diagrams in $\bA$, which can be checked straightforwardly.

\begin{lemma}\label{lem-isotopy}
Let $m,n\in\BN$ and $t\in\BZ$ (with $t=0$ if $n=0$).
Then the following planar isotopies hold:
\begin{align*}
\begin{array}{c}
    \includegraphics[scale=1.5]{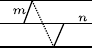}
\end{array}
&\;\sim\;
\begin{array}{c}
    \includegraphics[scale=1.5]{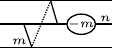}
\end{array},\\[1ex]
\begin{array}{c}
    \includegraphics[scale=1.5]{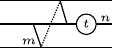}
\end{array}
&\;\sim\;
\begin{array}{c}
    \includegraphics[scale=1.5]{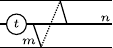}
\end{array}.
\end{align*}
Here the orientations of the diagrams are arbitrary (there are exactly two opposite choices for each diagram), provided that the orientations are compatible and are preserved within each isotopy.
\end{lemma}

We define \emph{line-circle webs} in $\bA$ or $\bA'$
to be the web diagrams illustrated in Figure~\ref{fig:line-circle-web}.

\begin{figure}[h]
    \centering
    \includegraphics[width=0.6\linewidth]{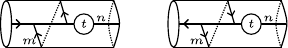}
    \caption{Line-circle webs in $\bA$ or $\bA'$, where $m$ is a positive integer, $n\in\mathbb N$ and $t\in\mathbb Z$.
    We require that $t=0$ if $n=0$.}
    \label{fig:line-circle-web}
\end{figure}

When $n=0$, the line-circle web in Figure~\ref{fig:line-circle-web} consists of $m$ parallel copies of an oriented closed curve.

The following will be used in Definition~\ref{def-twist-number} to define the coordinates of non-elliptic braided webs in
$\bA'$.

\begin{lemma}\label{lem-A}
    Let $W$ be a non-elliptic braided web in $\bA'$. Then:
    \begin{enumerate}[label={\rm (\alph*)}, itemsep=0.3em]
        \item $W$ is either a strict minimal braided web or a line-circle web.

        \item If $W$ is a strict minimal braided web, then there exists a unique strict minimal braid $\mathcal B$ such that $W=H(\cB)$. 
    \end{enumerate}
\end{lemma}

\begin{proof}
(a) 
If $W$ has a strict braid representative, then this representative is
minimal by Lemma~\ref{lem-minimal-non}(b), so $W$ is a strict minimal
braided web. Assume therefore that $W$ has no strict braid representative.
It then suffices to show that $W$ is a line-circle web.

Let $\cB$ be a braid representative of $W$.
Lemma~\ref{lem-minimal-non}(b) implies that $\cB$ is minimal.
Since $W$ has no strict braid representative, $\cB$ contains at least one
oriented closed curve $c$.

We first show that $\cB$ cannot contain both increasing and decreasing curves.
Suppose, to the contrary, that it does.
We will construct a strict minimal braid $\cB'$ such that $H(\cB')=W$.

Let $C$ be a closed curve in $\bA'$ parallel to a boundary component.
Recall the graded algebra $\cS_C(\bA')$ defined in \S\ref{sub-graded}.
It is straightforward to check that the two braids shown in Figure~\ref{fig:reorder}
represent the same element of $\cS_C(\bA')$.
By repeatedly applying the move in Figure~\ref{fig:reorder} to $\cB$, we obtain a new braid $\cB_1$ such that
\[
[H(\cB)] = [H(\cB_1)] \in \cS_C(\bA'),
\]
and such that the local configuration of $\cB_1$ around each closed curve is as in Figure~\ref{fig:ordering}.

Next, note that the two braids in Figure~\ref{fig:replace1}
(resp.\ Figure~\ref{fig:replace2}) represent the same element of $\cS(\bA')$.
Applying the moves in Figures~\ref{fig:replace1} and~\ref{fig:replace2} to each closed curve in $\cB_1$,
we obtain a new braid $\cB_2$ such that
\[
H(\cB_1)=H(\cB_2)\in \cS(\bA'),
\]
and $\cB_2$ is strict.
Let $\cB'$ be the minimal strict braid obtained from $\cB_2$ by removing all bigons.
Then
\[
[W]=[H(\cB)]=[H(\cB_1)]=[H(\cB_2)]=[H(\cB')]\in \cS_C(\bA').
\]
By Lemma~\ref{lem-graded}, this implies that $W=H(\cB')$,
contradicting the assumption that $W$ has no strict braid representative.

Therefore, $\cB$ cannot contain both increasing and decreasing curves.
The claim now follows from Lemma~\ref{lem-isotopy}.

\smallskip
(b) Suppose that $W=H(\cB)=H(\cB')$ for two strict minimal braids $\cB$ and $\cB'$.
Lemma~\ref{lem-flip} implies that $H(\cB)$ and $H(\cB')$ differ by a planar isotopy.
Consequently, $\cB$ and $\cB'$ themselves are related by a planar isotopy, and hence $\cB=\cB'$.
\end{proof}

\begin{figure}
    \centering
    \includegraphics[width=0.5\linewidth]{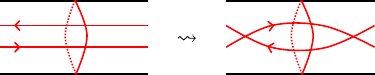}
    \caption{A reordering move for curves intersecting a closed curve in $\cB$.
    The orientation of the closed curve is arbitrary.}
    \label{fig:reorder}
\end{figure}

\begin{figure}
    \centering
    \includegraphics[width=0.2\linewidth]{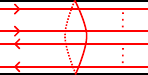}
    \caption{The local configuration of $\cB_1$ around each closed curve.}
    \label{fig:ordering}
\end{figure}

\begin{figure}
    \centering
    \includegraphics[width=0.5\linewidth]{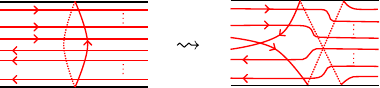}
    \caption{A move eliminating closed curves in the braid.}
    \label{fig:replace1}
\end{figure}

\begin{figure}
    \centering
    \includegraphics[width=0.5\linewidth]{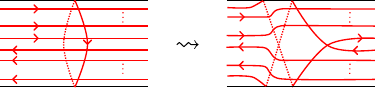}
    \caption{Another move eliminating closed curves in the braid.}
    \label{fig:replace2}
\end{figure}

The following proposition  shows that any non-elliptic web in $\bA$ is in fact a non-elliptic braided web.
Although it can also be formulated for a standalone annulus,
we use the ambient surface formulation for its later
applications to pants decompositions.
This proposition guarantees that the twist numbers defined in Definition~\ref{def-twist-number} can be used to define twist numbers
for non-elliptic web diagrams on $\fS_g$, associated to each annulus in a pants decomposition of $\fS_g$.
Moreover, this allows us to apply Theorem~\ref{prop-twist} to non-elliptic web diagrams on $\fS_g$,
annulus by annulus, with respect to a pants decomposition of $\fS_g$.

\begin{proposition}\label{prop-bigon}
Let $\fS$ be a marked surface,
let $W$ be a crossingless web diagram in $\Sigma$, and let $\mathbb{A}$ be an embedded annulus in $\Sigma$ with boundary components $C_1$ and $C_2$.
Suppose that $W$ is in minimal intersection position with respect to $C_1$ and $C_2$.
Then:
\begin{enumerate}[label={\rm (\alph*)}, itemsep=0.3em]

 \item If  $W\cap\bA$ contains no trivial loops, 2-gons, or 4-gons in \eqref{elliptic-face}, the web $W \cap \mathbb{A}$ is a non-elliptic braided web.

\item If $C_1$ and $C_2$ are oriented in the same direction, then
\begin{align}\label{eq-inter-numm}
     i_1(C_1, W) = i_1(C_2, W)
        \quad \text{and} \quad
        i_2(C_1, W) = i_2(C_2, W)\qquad
        (\text{see \eqref{def-i+-}}).
\end{align}


\end{enumerate}
\end{proposition}
\begin{proof}

(a) Set $W' := W \cap \mathbb A$.
We regard $\mathbb A$ as the twice-punctured sphere $\fS_{0,2}$ by viewing the two boundary components of $\bA$ as punctures.
Then $W'$ is an unbounded lamination in $\fS_{0,2}$, since $W$ is in minimal intersection position with respect to $\partial \mathbb A$.
Let $\mathcal W$ be the spiralling diagram of $W'$ (Definition~\ref{def-spiralling}).
By \cite[Proposition~6.4 and Lemma~6.5]{ishibashi2025unbounded}, there exists an ideal arc $E$ in $\fS_{0,2}$ such that, after cutting $\mathcal W$ along $E$, we obtain an unbounded essential web $\mathcal W'$ in the bigon $\mathbb D_2$ (Definition~\ref{def-essential-bigon}).

Choose a compact strip $K$ as in Definition~\ref{def-bigon}
whose boundary intervals correspond under the gluing
identification $F\colon E_L\to E_R$.
We label the oriented strands in $S_L$ and $S_R$ by
$l_i$ and $r_i$, respectively, for $i\in\mathbb Z$, so that
\[
K\cap S_L=\{l_1,\ldots,l_k\},\qquad
K\cap S_R=\{r_1,\ldots,r_k\},\qquad
F(l_i)=r_i.
\]
Then
\[
\fS_{0,2}
= \mathbb D_2/(F(x)\sim x,\ x\in E_L),
\qquad
\mathcal W
= \mathcal W'/(l_i\sim r_i,\ i\in\mathbb Z).
\]


Let $E_L'$ (resp. $E_R'$) be a closed subinterval of $E_L$ (resp. $E_R$) such that
\[
E_L' \cap S_L = \{l_1,\ldots,l_k\}, \qquad
E_R' \cap S_R = \{r_1,\ldots,r_k\}.
\]
We may assume that $F(E_L') = E_R'$, since otherwise $F$ is isotopic to an identification with this property.
Let $\mathcal W''$ be the subweb of $\mathcal W'$
consisting of components all of whose endpoints lie in
$E_L'\sqcup E_R'$.
This restriction discards the infinite peripheral portions
introduced by the spiralling construction.
Let $\mathcal B$ be the braid representative of $\mathcal W''$.

Define
\[
\bA' = \mathbb D_2 / (F(x)\;\sim\;x,\ x \in E_L'), \qquad
\overline W = \mathcal W'' / (l_i \;\sim\; r_{i},\ 1 \le i \le k),
\]
and
\[
\overline{\mathcal B}
= \mathcal B / (l_i \;\sim\; r_{i},\ 1 \le i \le k).
\]
Then $\mathbb A'$ is obtained from $\bA$ by removing one point from each boundary component.
Viewed as a web diagram in $\bA$ via the embedding $\mathbb A' \subset \bA$, the web $\overline W$ differs from $W'$ by a sequence of twists of the endpoints of $W$ along $A_0$ or $A_1$ (see \eqref{A01}).
Moreover, $\overline{\mathcal B} \subset \bA' \subset \bA$ is a braid in $\bA$ 
because the strands in $\mathcal B$ with the same direction do not intersect each other.
We also have
$H(\overline{\mathcal B}) = \overline W$.
Hence $\overline W$ is a braided web, and therefore so is $W'$.
Lemma~\ref{lem-minimal-non}(a) shows that $W'$ is a non-elliptic braided web.


(b) follows from (a), because removing trivial loops, $2$-gons, or $4$-gons (Figure~\ref{fig:H-annulus}) in $W \cap \bA$ does not affect the intersection numbers in \eqref{eq-inter-numm}.

\end{proof}

\begin{figure}
    \centering
    \includegraphics[width=0.5\linewidth]{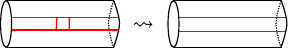}
    \caption{The procedure to remove 4-gons in the annulus.}
    \label{fig:H-annulus}
\end{figure}

Using the technique of \cite[Proposition~13]{frohman20223}, we obtain the following corollary from
Lemma~\ref{lem-flip} and Proposition~\ref{prop-bigon}.

\begin{figure}
    \centering
    \includegraphics[width=0.5\linewidth]{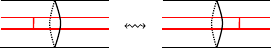}
    \caption{The crossbar move.}
    \label{fig:cb}
\end{figure}

\begin{figure}
    \centering
    \includegraphics[width=0.5\linewidth]{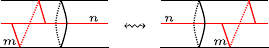}
    \caption{The annulus-H move.}
    \label{fig:H-moves}
\end{figure}

\begin{corollary}\label{Cor-minimal-unique}
Let $\fS$ be a marked surface.
The minimal position of any non-elliptic web diagram with respect to any disjoint collection of closed curves
$C_1,\ldots,C_n$ in $\Sigma$ is unique up to crossbar moves (Figure~\ref{fig:cb}), annulus-H moves (Figure~\ref{fig:H-moves}), twists along some of these
closed curves, flip moves (Figure~\ref{fig:flip}), and planar isotopy within the complement of these curves.
\end{corollary}

\begin{proof}
\textbf{Step 1: The case $n=1$.}
This follows immediately from Lemma~\ref{lem-flip}, Proposition~\ref{prop-bigon}(a), and Lemma~\ref{lem-A}(a).

\textbf{Step 2: The general case.}
We argue by induction on $n$.
Let $C_1,\ldots,C_{n+1}$ be $n+1$ disjoint closed curves in $\Sigma$, and let $W$ and $W'$ be two isotopic
non-elliptic web diagrams that are both in minimal position with respect to $C_1,\ldots,C_{n+1}$.
By Step~1, the webs $W$ and $W'$ can be made isotopic in $\Sigma \setminus C_{n+1}$ using crossbar moves,
annulus-H moves, twists along $C_{n+1}$, and flip moves.
The desired conclusion then follows from the inductive hypothesis.
\end{proof}

Definition~\ref{def-general-position} and Corollary~\ref{Cor-minimal-unique} together imply the following result, which will be used in
\S\ref{Sec-coord} to define coordinates for non-elliptic web diagrams on a closed surface.

\begin{corollary}\label{cor-position-unique}
Let $\mathcal P$ be a pants decomposition of $\Sigma_g$ with $g\geq 2$.
A non-elliptic web diagram $W$ in general position with respect to $\mathcal P$ is unique up to crossbar path moves (Figure~\ref{fig:crross-bar-pass}) and flip moves (Figure~\ref{fig:flip}).
\end{corollary}

\begin{figure}[H]
    \centering
    \includegraphics[width=0.5\linewidth]{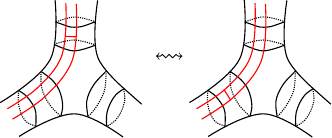}
    \caption{The crossbar path move.}
    \label{fig:crross-bar-pass}
\end{figure}

\subsection{Twist numbers}

In this subsection, we will define the coordinates of non-elliptic braided webs in $\bA'$. 

Let $\cB=\{c_1,\ldots,c_n\}$ be a strict braid in $\bA'$.
We define
\begin{align*}
    n_1(\cB) &:= \text{the number of increasing curves in $\cB$},\\
    n_2(\cB) &:= \text{the number of decreasing curves in $\cB$},\\
    t_1(\cB) &:= \sum_{c_i\ \text{increasing } \text{in }\cB} t(c_i),\\
    t_2(\cB) &:= \sum_{c_i\ \text{decreasing } \text{in }\cB} t(c_i),
\end{align*}
where the twist number $t(c_i)$ is defined as in Figure~\ref{fig:pos-neg}.

\begin{definition}\label{def-twist-number}
Let $W$ be a non-elliptic braided web in $\mathbb{A}'$. 
Lemma~\ref{lem-A}(a) implies that $W$ is either a strict minimal braided web or a line-circle web.
We define 
$${\bf t}(W) = (n_1(W), n_2(W), t_1(W), t_2(W))$$
as follows:
\begin{itemize}
    \item If $W$ is a line-circle web as illustrated in the left picture in Figure~\ref{fig:line-circle-web}, define
    \[
        {\bf t}(W):=(n,0,t,m).
    \]

    \item If $W$ is a line-circle web as illustrated in the right picture in Figure~\ref{fig:line-circle-web}, define
    \[
        {\bf t}(W):=(0,n,m,t).
    \]

    \item If $W$ is a strict minimal braided web represented by a strict minimal braid $\cB$, define
    \[
        {\bf t}(W):=\bigl(n_1(\cB),\,n_2(\cB),\,t_1(\cB),\,t_2(\cB)\bigr).
    \]
\end{itemize}
\end{definition}

Lemma~\ref{lem-A}(b) implies the following result.

\begin{lemma}
Let $W$ be a non-elliptic braided web in $\mathbb{A}'$.
Then ${\bf t}(W)$ is well defined.
\end{lemma}

Let $C=S^1\times \{\frac{1}{2}\}$. We orient $C$ as illustrated in Figure~\ref{fig:curve-C}.
It is straightforward to check that 
$$n_1(W) = i_1(C,W),\qquad
 n_2(W) = i_2(C,W).$$

\begin{figure}[H]
    \centering
    \includegraphics[width=0.2\linewidth]{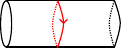}
    \caption{The oriented red closed curve is $C$ (the positive orientation of $[0,1]$ is from left to right).}
    \label{fig:curve-C}
\end{figure}

We orient $A_0,A_1$ (see \eqref{A01}) in the same direction as $C$.

Let \(W\) be a braided web in \(\mathbb{A}'\), and suppose that
\[
W \cap \partial A_0 = \{a_1,\ldots,a_n\}, \qquad
W \cap \partial A_1 = \{a_{n+1},\ldots,a_{2n}\}.
\]
Assume that the points \(a_1,\ldots,a_n\) (resp. \(a_{n+1},\ldots,a_{2n}\)) are encountered consecutively when traversing \(A_0\) (resp. \(A_1\)) in its orientation, starting from the unique puncture in \(A_0\) (resp. \(A_1\)).

For each \(1 \le i \le 2n\), define a map
\[
s \colon \{a_1,\ldots,a_{2n}\} \longrightarrow \{-,+\}
\]
by
\[
s(a_i)=
\begin{cases}
+ & \text{if the strand incident to \(a_i\) is increasing},\\
- & \text{if the strand incident to \(a_i\) is decreasing}.
\end{cases}
\]
We call the tuple 
\begin{align}\label{def-signature}
   {\bf s}(W)= \bigl(s(a_1),\ldots,s(a_{2n})\bigr)
\end{align}
the \emph{signature} of \(W\).

\def\bt{{\bf t}}

The following theorem is the main result of this section. It shows that a non-elliptic braided web in $\mathbb{A}'$ is uniquely determined by its coordinates, as defined in Definition~\ref{def-twist-number}, together with its signature.

\begin{theorem}\label{prop-twist}
Let $W$ be a non-elliptic braided web in $\mathbb{A}'$.
Then $W$ is uniquely determined by its coordinate
\[
\bt(W) = (n_1,n_2,t_1,t_2)
\]
together with its signature.
Moreover, we have
\begin{align*}
    \im\bt = \{(n_1,n_2,t_1,t_2)\in\BN^2\times \BZ^2\mid t_i\geq 0 \text{ if }n_i=0, \text{ for }i=1,2\}.
\end{align*}
\end{theorem}
\begin{proof}

Lemma~\ref{lem-A}(a) implies that $W$ is either a strict minimal braided web or a line-circle web.
It follows immediately that
\[
\im \bt
=
\bigl\{
(n_1,n_2,t_1,t_2) \in \BN^2 \times \BZ^2
\;\big|\;
t_i \ge 0 \text{ if } n_i = 0,\ i=1,2
\bigr\}.
\]

Moreover, it is straightforward to see that
\[
{\bf t}(W_1) \neq {\bf t}(W_2)
\]
whenever $W_1$ is a strict minimal braided web and $W_2$ is a line-circle web,
or when $W_1 \neq W_2$ are both line-circle webs.

Suppose that $W$ is a strict minimal braided web, represented by a strict minimal braid $\cB$.
The signature of $W$, together with the data $(n_1,t_1)$ (resp.\ $(n_2,t_2)$), uniquely determines the increasing (resp.\ decreasing) curves in $\cB$.
The minimality of $\cB$ then uniquely determines $\cB$.
Hence $W$ is uniquely determined.
\end{proof}

\section{Webs in a pair of pants}\label{Sec-pants}

\def\sph{\Sigma_{0,3}}
\def\cL{\mathcal{L}}
\def\cll{\mathcal L(\Sigma_{0,3})}
\def\bx{{\bf x}}

Recall that $P'$ is the marked surface obtained from the pair of pants $P$ by removing one point from each boundary component.
In this section, we define coordinates for essential non-elliptic web diagrams (Definition~\ref{def-essential-pants}) in the pair of pants $P'$; see Definition~\ref{def:coordinates-pants}.
The main result of this section is Theorem~\ref{lem:image-P}, which asserts that a non-elliptic web diagram in $P'$, up to the moves shown in Figure~\ref{fig:flip-pants}, is uniquely determined by its coordinates.

Let $W$ be an essential non-elliptic web diagram in $P'$.
Via the natural embedding $P' \hookrightarrow P$, we regard $W$ as a web diagram in $P$.
We identify $P$ with $\Sigma_{0,3}$ as in Figure~\ref{fig:sphere}.
By Definition~\ref{def-essential-pants}, it follows that
$W \in \mathcal{L}(\Sigma_{0,3}).$

We label $I(\Sigma_{0,3})$ as in Figure~\ref{fig:sphere}.
Suppose that
\[
    \bx(W)
    = (x_{11}, x_{12}, x_{21}, x_{22}, x_{31}, x_{32}, x_v, x_{v'})
    \in \mathbb{Z}^{I(\Sigma_{0,3})} .
\]

We introduce the following coordinates associated to $W$.

\def\oB{\overline{B}}
\def\bP{{\bf P}}

\begin{definition}
\label{def:coordinates-pants}
Let $\oB$ denote the set of essential non-elliptic web diagrams in $P'$.
Define
\begin{align*}
    {\bf P}\colon \oB\rightarrow \mathbb{Z}^{I(\Sigma_{0,3})},\qquad
    W\mapsto (n_{11}(W),n_{12}(W),n_{21}(W),n_{22}(W),n_{31}(W), n_{32}(W),t_P(W),h_P(W)),
\end{align*}
where
\begin{equation}
\label{eq:coordinates-pants}
\begin{aligned}
    n_{11}(W) &= x_{11} + x_{32},\\
    n_{12}(W) &= x_{12} + x_{31} + x_v + x_{v'},\\
    n_{21}(W) &= x_{31} + x_{22},\\
    n_{22}(W) &= x_{32} + x_{21} + x_v + x_{v'},\\
    n_{31}(W) &= x_{21} + x_{12},\\
    n_{32}(W) &= x_{22} + x_{11} + x_v + x_{v'},\\
    t_P(W) &= x_v - x_{v'},\\
    h_P(W) &= x_{11} - x_{12} + x_{21} - x_{22} + x_{31} - x_{32}.
\end{aligned}
\end{equation}
\end{definition}

Label the three boundary components of $P'$ by $C_1, C_2, C_3$ as in
Figure~\ref{fig:sphere}.
For $t=1,2,3$, define the intersection numbers
\begin{equation}
\label{eq:intersection-numbers}
\begin{aligned}
    i_{t1}(W) &:= \text{the number of endpoints of $W$ on $C_t$ pointing towards $C_t$},\\
    i_{t2}(W) &:= \text{the number of endpoints of $W$ on $C_t$ pointing away from $C_t$}.
\end{aligned}
\end{equation}

The following provides a geometric interpretation of the coordinates $n_{t1}(W),\, n_{t2}(W)$, for $t=1,2,3$, in \eqref{eq:coordinates-pants}; namely, they coincide with the intersection numbers defined in \eqref{eq:intersection-numbers}.

\begin{proposition}
\label{prop:intersection-coordinates}
Let $W$ be an essential non-elliptic web diagram in $P'$.
Then, for $t=1,2,3$, we have
\[
    i_{t1}(W) = n_{t1}(W),
    \qquad
    i_{t2}(W) = n_{t2}(W).
\]
\end{proposition}

\begin{proof}
Identify the pair of pants $P$ with $\sph$ and regard $W$
as an unbounded $\SL$-lamination on $\sph$.
For each $t=1,2,3$, let $p_t$ be the puncture corresponding
to the boundary component $C_t$.
By \cite[Proposition~4.4]{ishibashi2024unbounded}, the
combinations of shear coordinates defining $n_{t1}(W)$
and $n_{t2}(W)$ are the two components of the tropicalized
Casimir map $\theta_{p_t}(W)$.
Our unbounded $\SL$-lamination corresponds to an integral unbounded 
$\SL$-lamination in \cite{ishibashi2024unbounded} with a
positive sign at every endpoint; see
\cite[Figure~3]{ishibashi2024unbounded}.
Consequently, the two components of $\theta_{p_t}(W)$
count the ends pointing towards and away from $p_t$,
respectively.
These counts are precisely $i_{t1}(W)$ and $i_{t2}(W)$.
Thus
\[
    n_{tj}(W)=i_{tj}(W),
    \qquad t=1,2,3,\quad j=1,2.
\]
\end{proof}

\begin{figure}[H]
    \centering
    \includegraphics[width=0.5\linewidth]{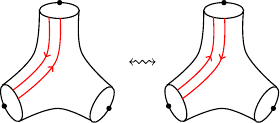}
    \caption{The flip move in the pair of pants.}
    \label{fig:flip-pants}
\end{figure}

We are now ready to state the main theorem in this section. 

\begin{theorem}\label{lem:image-P}
(a)   Let $W$ and $W'$ be two essential non-elliptic web diagrams in $P'$ with
 $\bP(W)=\bP(W')$.
Then $W$ can be obtained from $W'$ by a sequence of moves shown in Figure~\ref{fig:flip-pants}, together with the moves obtained by rotating the pictures in Figure~\ref{fig:flip-pants} by $120^\circ$ and $240^\circ$.

(b)    The image $\im \bP$ consists of all tuples
    \[
    (n_{11}, n_{12}, n_{21}, n_{22}, n_{31}, n_{32}, t_P, h_P)
    \in \mathbb{Z}^{I(\Sigma_{0,3})}
    \]
    satisfying the following conditions:
    \begin{align}\label{eq-domain-P}
        \begin{cases}
            n_{11} \ge 0,\; n_{12} \ge 0,\; n_{21} \ge 0,\; n_{22} \ge 0,\;
            n_{31} \ge 0,\; n_{32} \ge 0, \\[4pt]
            n_{12} + n_{22} + n_{32} \equiv
            n_{11} + n_{21} + n_{31} \pmod{3}, \\[4pt]
            h_P + n_{11} + n_{21} + n_{31} \equiv 0 \pmod{2}, \\[4pt]
            h_P \equiv
            n_{11} - n_{12}-n_{21} + n_{22} \pmod{3}, \\[4pt]
            3t_P \equiv n_{12} + n_{22} + n_{32}
            - n_{11} - n_{21} - n_{31} \pmod{6}.
        \end{cases}
    \end{align}

\end{theorem}

\begin{proof}
Observe that the equations in \eqref{eq:coordinates-pants} are linearly independent. 
Their inverse transformation is given by

\begin{align}
\nonumber
    x_{11} &= \frac{h_P + 3n_{11} - n_{21} - 2n_{22} + n_{31} + 2n_{32}}{6},\\
    \nonumber
x_{12} &= \frac{-h_P + n_{11} + 2n_{12} - n_{21} - 2n_{22} + 3n_{31}}{6},\\
\nonumber
x_{21} &= \frac{h_P - n_{11} - 2n_{12} + n_{21} + 2n_{22} + 3n_{31}}{6},\\
\nonumber
x_{22} &= \frac{-h_P - n_{11} - 2n_{12} + 3n_{21} + n_{31} + 2n_{32}}{6},\\
\label{eq-inverse}
x_{31} &= \frac{h_P + n_{11} + 2n_{12} + 3n_{21} - n_{31} - 2n_{32}}{6},\\
\nonumber
x_{32} &= \frac{-h_P + 3n_{11} + n_{21} + 2n_{22} - n_{31} - 2n_{32}}{6},\\
\nonumber
x_v &= \frac{3t_P - n_{11} + n_{12} - n_{21} + n_{22} - n_{31} + n_{32}}{6},\\
\nonumber
x_{v'} &= \frac{-3t_P - n_{11} + n_{12} - n_{21} + n_{22} - n_{31} + n_{32}}{6}.
\end{align}

For an integer tuple
$(n_{11},n_{12},n_{21},n_{22},n_{31},n_{32},t_P,h_P)$,
the four congruence conditions in \eqref{eq-domain-P} are equivalent to
the integrality of all eight coordinates defined by \eqref{eq-inverse}.

Part~(a) follows directly from Lemma~\ref{lem-domain}.

For part~(b), using the equations in \eqref{eq:coordinates-pants} together with Lemma~\ref{lem-domain}, one readily checks that any element
\[
(n_{11}, n_{12}, n_{21}, n_{22}, n_{31}, n_{32}, t_P, h_P)\in \im \bP
\]
satisfies the relations in \eqref{eq-domain-P}.

Conversely, let $(n_{11}, n_{12}, n_{21}, n_{22}, n_{31}, n_{32}, t_P, h_P)\in \BZ^{I(\Sigma_{0,3})}$ satisfy \eqref{eq-domain-P}.
Define 
$$(x_{11}, x_{12}, x_{21}, x_{22}, x_{31}, x_{32}, x_v, x_{v'})\in \mathbb Q^{I(\Sigma_{0,3})}$$
by the formulas in \eqref{eq-inverse}.
It is straightforward to verify that
\[
(x_{11}, x_{12}, x_{21}, x_{22}, x_{31}, x_{32}, x_v, x_{v'}) \in \Lambda
\]
(see \eqref{def-lambda}).
By Lemma~\ref{lem-domain}, there exists $W\in \oB$ such that
\[
\bx(W) = (x_{11}, x_{12}, x_{21}, x_{22}, x_{31}, x_{32}, x_v, x_{v'}).
\]
Applying $\bP$, we obtain
\[
\bP(W) = (n_{11}, n_{12}, n_{21}, n_{22}, n_{31}, n_{32}, t_P, h_P),
\]
which completes the proof.
\end{proof}

\begin{remark}\label{rem-inva}
Note that the congruence
\[
n_{12} + n_{22} + n_{32} \equiv
n_{11} + n_{21} + n_{31} \pmod{3}
\]
implies
\begin{align*}
    n_{11}- n_{12} -n_{21}  + n_{22}
    \equiv
    n_{21}- n_{22} -n_{31}  + n_{32} 
    \equiv
    n_{31}- n_{32} -n_{11} + n_{12}
    \pmod{3}.
\end{align*}
Consequently, condition~\eqref{eq-domain-P} is invariant under the cyclic permutation
\[
n_{1j}\mapsto n_{2j}\mapsto n_{3j}\mapsto n_{1j}.
\]
\end{remark}

\def\bP{\mathbb{P}}
\def\cN{\mathcal {N}}

\section{Coordinates for non-elliptic ${\rm SL}_3$ web diagrams on closed surfaces}\label{Sec-coord}

In this section, we use the results developed in \S\ref{Sec-annulus} and \S\ref{Sec-pants} to construct coordinates for non-elliptic web diagrams on a closed surface $\fS_g$.
We divide the construction into two cases: the case $g\geq 2$ (Definition~\ref{def-DT-coordinates}) and the case $g=1$ (Definition~\ref{def-closed-torus}).
We then prove that these coordinates uniquely determine a non-elliptic web diagram (Theorems~\ref{thm:DT-coordinates} and \ref{thm-torus}), which constitute the main results of this paper.
As a consequence, we obtain a parametrization of the non-elliptic basis elements of $\cS(\fS_g)$.

\subsection{The genus is more than one}

 Let $\mathcal P=\{C_j\}_{1\leq j\leq 3g-3}$ be an oriented pants decomposition of $\Sigma_g$ $(g\geq 2)$ with dual graph $\Gamma$, and let $W\in B_{\fS}$ be a non-elliptic web diagram in $\fS_g$ in general position with respect to $\{C_j\}_{1\leq j\leq 3g-3}$.
We write $r=3g-3$, let $\mathbb P$ denote the set of pairs of pants determined by $\{C_j\}_{1\leq j\leq 3g-3}$, and for each $j$ let
\[
\cN_j := N(C_j)\setminus \bigl(\partial N(C_j)\cap \Gamma\bigr).
\]
We identify each $\cN_j$ with $\bA'$ by an
orientation preserving homeomorphism sending the oriented
curve $C_j$ to the oriented curve $C$ in
Figure~\ref{fig:curve-C}.
The increasing direction on $\cN_j$ is the direction
corresponding to the positive direction of the $[0,1]$ factor
of $\bA'$.
This direction is determined by the orientation of $\Sigma_g$
together with the fixed orientation of $C_j$; thus no
additional choice of increasing direction is required.

\begin{definition}\label{def-DT-coordinates}
We define the coordinate of $W$ (with respect to $\cP$) by
\begin{align*}
\kappa(W)
= \Bigl(
    (n_{j1}(W))_{1\leq j\leq r},\,
    (n_{j2}(W))_{1\leq j\leq r},\,
    (t_{j1}(W))_{1\leq j\leq r},\,
    (t_{j2}(W))_{1\leq j\leq r},\,
    (t_P(W))_{P\in \bP},\,
    (h_P(W))_{P\in \bP}
\Bigr),
\end{align*}
where
\begin{align*}
    n_{j1}(W) &= i_1(C_j, W) \quad \text{(see \eqref{def-i+-}), for $1\leq j\leq r$},\\
    n_{j2}(W) &= i_2(C_j, W) \quad \text{(see \eqref{def-i+-}), for $1\leq j\leq r$},\\
    t_{j1}(W) &= t_1(\cN_j\cap W) \quad \text{(see Definition~\ref{def-twist-number}), for $1\leq j\leq r$},\\
    t_{j2}(W) &= t_2(\cN_j\cap W) \quad \text{(see Definition~\ref{def-twist-number}), for $1\leq j\leq r$},\\
    t_P(W) &= t_P(P\cap W) \quad \text{(see Definition~\ref{def:coordinates-pants}), for $P\in\bP$},\\
    h_P(W) &= h_P(P\cap W) \quad \text{(see Definition~\ref{def:coordinates-pants}), for $P\in\bP$}.
\end{align*}
\end{definition}

We have the following. 

\begin{lemma}
The coordinates defined in Definition~\ref{def-DT-coordinates} are well defined.
\end{lemma}

\begin{proof}
The well-definedness of $t_P$, $h_P$, $t_{i1}$, and $t_{i2}$, follows from Corollary~\ref{cor-position-unique}.
The well-definedness of $n_{i1}$ and $n_{i2}$ follows from  Proposition~\ref{prop-bigon}(b).
\end{proof}

Define
\[
K := \BN^r \times \BN^r \times \BZ^r \times \BZ^r \times \BZ^{\bP} \times \BZ^{\bP}.
\]
Let
\[
\Bigl(
    (n_{i1})_{1\leq i\leq r},\,
    (n_{i2})_{1\leq i\leq r},\,
    (t_{i1})_{1\leq i\leq r},\,
    (t_{i2})_{1\leq i\leq r},\,
    (t_P)_{P\in \bP},\,
    (h_P)_{P\in \bP}
\Bigr) \in K.
\]

We use Figure~\ref{fig:placeholder} to illustrate the clockwise orientation of the three boundary components of the pair of pants.
The counterclockwise orientation is defined to be the opposite one. 

\begin{figure}
    \centering
    \includegraphics[width=0.2\linewidth]{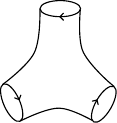}
    \caption{The clockwise orientation of the three boundary components of the pair of pants.}
    \label{fig:placeholder}
\end{figure}

For each $P\in \bP$, suppose that $P$ is bounded by the curves $C_j$, $C_k$, and $C_m$.
We identify $P$ with the pair of pants in Figure~\ref{fig:sphere} so that $C_j$, $C_k$, and $C_m$ correspond to $C_1$, $C_2$, and $C_3$ in Figure~\ref{fig:sphere}, respectively.
For a boundary component of $P$, if $C_j$ (resp. $C_k$ or $C_m$) is oriented counterclockwise  (resp. clockwise), we define
\begin{align}
    &n_{11}^P = n_{j1}, \quad n_{12}^P = n_{j2}
    \qquad
    (\text{resp. } n_{11}^P = n_{j2}, \; n_{12}^P = n_{j1}),\nonumber\\
    &n_{21}^P = n_{k1}, \quad n_{22}^P = n_{k2}
    \qquad
    (\text{resp. } n_{21}^P = n_{k2}, \; n_{22}^P = n_{k1}),
    \label{def-nij-P}\\
    &n_{31}^P = n_{m1}, \quad n_{32}^P = n_{m2}
    \qquad
    (\text{resp. } n_{31}^P = n_{m2}, \; n_{32}^P = n_{m1}).\nonumber
\end{align}

Although the quantities $n_{1t}^P, n_{2t}^P, n_{3t}^P$ for $t=1,2$ are not individually well defined, they are well defined up to the cyclic permutation
\[
n_{1t}\mapsto n_{2t}\mapsto n_{3t}\mapsto n_{1t}.
\]

\begin{definition}\label{def-image-theta}
Define the submonoid $\Theta\subset K$ to be the set of all
\[
\Bigl(
    (n_{i1})_{1\leq i\leq r},\,
    (n_{i2})_{1\leq i\leq r},\,
    (t_{i1})_{1\leq i\leq r},\,
    (t_{i2})_{1\leq i\leq r},\,
    (t_P)_{P\in \bP},\,
    (h_P)_{P\in \bP}
\Bigr) \in K
\]
satisfying the following conditions:
\begin{enumerate}
    \item For each $1\leq i\leq r$, we have $t_{i1}\geq 0$ (resp.\ $t_{i2}\geq 0$) whenever $n_{i1}=0$ (resp.\ $n_{i2}=0$).
    
    \item For each $P\in \bP$, the following congruences hold:
    \begin{align*}
        \begin{cases}
            n_{12}^P + n_{22}^P + n_{32}^P \equiv
            n_{11}^P + n_{21}^P + n_{31}^P \pmod{3}, \\[4pt]
            h_P + n_{11}^P + n_{21}^P + n_{31}^P \equiv 0 \pmod{2}, \\[4pt]
            h_P \equiv
            n_{11}^P + 2n_{21}^P - n_{12}^P + n_{22}^P \pmod{3}, \\[4pt]
            3t_P \equiv
            n_{12}^P + n_{22}^P + n_{32}^P
            - n_{11}^P - n_{21}^P - n_{31}^P \pmod{6}.
        \end{cases}
    \end{align*}
\end{enumerate}
By Remark~\ref{rem-inva}, the set $\Theta$ is well defined.
\end{definition}

The following theorem is our first main result. It shows that the coordinates defined in Definition~\ref{def-DT-coordinates} parametrize the non-elliptic web diagrams in $\fS_g$ for $g\geq 2$, and that the set of coordinates of all non-elliptic web diagrams in $\fS_g$ is precisely the monoid $\Theta$.
  
We next introduce a definition that will be used in the proof of this theorem.
A crossingless web diagram in $\fS_g$ is called \emph{almost non-elliptic} if it contains no $2$-gons and no trivial loops as in \eqref{elliptic-face}.

\begin{figure}[H]
    \centering
    \includegraphics[width=0.5\linewidth]{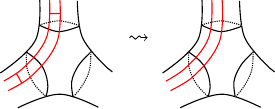}
    \caption{The procedure to remove 4-gons in $W'$.}
    \label{fig:remove-4-gon}
\end{figure}

\begin{theorem}\label{thm:DT-coordinates}
Let $\mathcal P=\{C_j\}_{1\leq j\leq r}$ be an oriented pants decomposition of $\Sigma_g$ $(g\geq 2)$ with a dual graph $\Gamma$, where $r=3g-3$.
The coordinate map
\[
\kappa \colon B_{\fS_g} \longrightarrow
\BN^r \times \BN^{r} \times \BZ^r \times \BZ^r \times \BZ^{\bP} \times \BZ^{\bP}
\]
is injective, where $B_{\fS_g}$ is defined in Definition~\ref{def-basis} and $\kappa$ is defined in Definition~\ref{def-DT-coordinates}.
Moreover,
\[
\im \kappa = \Theta.
\]
\end{theorem}

\begin{proof}

We first show that $\im \kappa = \Theta$.
By Proposition~\ref{prop:intersection-coordinates}, Theorems~\ref{prop-twist} and \ref{lem:image-P}, we have
\[
\im \kappa \subset \Theta .
\]
Now let
\[
\Bigl(
    (n_{i1})_{1\leq i\leq r},\,
    (n_{i2})_{1\leq i\leq r},\,
    (t_{i1})_{1\leq i\leq r},\,
    (t_{i2})_{1\leq i\leq r},\,
    (t_P)_{P\in \bP},\,
    (h_P)_{P\in \bP}
\Bigr) \in \Theta .
\]

For each $P \in \bP$, let
\[
P' = P \setminus (\partial P \cap \Gamma).
\]
By Proposition~\ref{prop:intersection-coordinates}, Theorem~\ref{lem:image-P}, and Definition~\ref{def-image-theta}, there exists an essential non-elliptic web diagram $W_P$ in $P'$ such that
\[
{\bf P}(W_P)
=
(n_{11}^P, n_{12}^P, n_{21}^P, n_{22}^P, n_{31}^P, n_{32}^P, t_P, h_P)
\in \mathbb{Z}^{I(\Sigma_{0,3})},
\]
where $n_{jt}^P$ is defined in \eqref{def-nij-P}, the labeling of $I(\Sigma_{0,3})$ is shown in Figure~\ref{fig:sphere}, and
${\bf P}$ is defined in Definition~\ref{def:coordinates-pants}.

For each $1 \leq i \leq r$, suppose that the annulus $N(C_i)$ is adjacent to the pairs of pants $P_1$ and $P_2$ (possibly $P_1 = P_2$).
We identify
\[
\cN_i := N(C_i) \setminus (\partial N(C_i) \cap \Gamma)
\]
with $\bA'$ (here $\bA'$ is obtained from the annulus by removing one point from each its boundary component), so that $C_i$ is identified with the curve $C$ in Figure~\ref{fig:curve-C}, with orientations preserved.
Then the quadruple $(n_{i1}, n_{i2}, t_{i1}, t_{i2})$ uniquely determines a non-elliptic braided web $W_i$ in $\cN_i$, up to signature, by Theorem~\ref{prop-twist}.
There is a unique choice of signature for $W_i$ that allows $W_i$ to be glued compatibly with
$W_{P_1}$ and $W_{P_2}$ when $\cN_i$ is glued to $P_1$ and $P_2$.

After gluing all annuli $N(C_i)$, $1 \leq i \leq r$, with all pairs of pants $P \in \bP$ to recover $\fS_g$,
and simultaneously gluing all webs $W_i$, $1 \leq i \leq r$, with the webs $W_P$, $P \in \bP$,
we obtain an almost non-elliptic web diagram $W'$ on $\fS_g$.
This web may contain $4$-gons, as illustrated in the left picture in Figure~\ref{fig:remove-4-gon}.
By removing these $4$-gons as in Figure~\ref{fig:remove-4-gon}, we obtain a non-elliptic web diagram $W$.
By construction, it is immediate that
\[
\kappa(W)
=
\Bigl(
    (n_{i1})_{1\leq i\leq r},\,
    (n_{i2})_{1\leq i\leq r},\,
    (t_{i1})_{1\leq i\leq r},\,
    (t_{i2})_{1\leq i\leq r},\,
    (t_P)_{P\in \bP},\,
    (h_P)_{P\in \bP}
\Bigr).
\]

We now prove that $\kappa$ is injective.
Suppose that $\kappa(W_1) = \kappa(W_0)$ for two non-elliptic web diagrams $W_1$ and $W_0$ in $\fS_g$ that are in general position with respect to $\cP$.
Recall the graded algebra $\cS_\cP(\fS_g)$ defined in \eqref{def-graded-algebra}.
To show that $W_1 = W_0$, it suffices to prove that (Lemma~\ref{lem-graded})
\[
[W_1] = [W_0] \in \cS_\cP(\fS_g).
\]

For each $P \in \bP$, Theorem~\ref{lem:image-P}(a) implies that $W_1 \cap P$ can be obtained from $W_0 \cap P$ by a sequence of moves shown in Figure~\ref{fig:flip-pants}, together with the moves obtained by rotating the pictures in Figure~\ref{fig:flip-pants} by $120^\circ$ and $240^\circ$.
By repeatedly applying moves in Figure~\ref{fig:moves-pants} to $W_1$, we obtain an almost non-elliptic web diagram $W_2$ on $\fS_g$ such that
$W_2 \cap P = W_0 \cap P$ for each $P \in \bP$.
Relations~\eqref{4-gon} and \eqref{4-gon-re} then imply that
\[
[W_1] = [W_2] \in \cS_\cP(\fS_g).
\]

Let $W_3$ be the web diagram obtained from $W_2$ by removing all $4$-gons in the annuli using the moves shown in Figure~\ref{fig:H-annulus}.
Similarly, we have
\[
[W_1] = [W_2] = [W_3] \in \cS_\cP(\fS_g).
\]
Since all $4$-gons in the annuli have been removed, Lemma~\ref{lem-minimal-non}(a) implies that
$W_3 \cap \cN_i$ is a non-elliptic
braided web.
Because $W_3 \cap P = W_0 \cap P$ for each $P \in \bP$, the webs
$W_3 \cap \cN_i$ and $W_0 \cap \cN_i$ have the same signature.
Moreover,
\[
{\bf t}(W_3 \cap \cN_i)
=
{\bf t}(W_1 \cap \cN_i)
=
{\bf t}(W_0 \cap \cN_i),
\]
where ${\bf t}$ is defined in Definition~\ref{def-twist-number}.
By Theorem~\ref{prop-twist}, it follows that
\[
W_3 \cap \cN_i = W_0 \cap \cN_i
\qquad \text{for all } 1 \leq i \leq r .
\]
Hence $W_3 = W_0$, and therefore
\[
[W_1] = [W_2] = [W_3] = [W_0] \in \cS_\cP(\fS_g).
\]
This completes the proof.

\end{proof}

\begin{figure}
    \centering
    \includegraphics[width=0.5\linewidth]{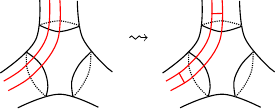}
    \caption{The procedure to obtain $W_2$.}
    \label{fig:moves-pants}
\end{figure}

\subsection{The closed torus}

Let $\gamma$ be an oriented closed curve in $\fS_1$, and fix a point $p \in \gamma$.
A non-elliptic web diagram $W$ on $\fS_1$ is said to be in \textbf{general position}
with respect to $\gamma$ if $W$ is in minimal intersection position with $\gamma$
and $p \notin W$.

Cutting $\fS_1$ along $\gamma$ and taking the closure yields an annulus, which we denote by $\bA_\gamma$.
The point $p$ gives rise to two points $p', p'' \in \partial \bA_\gamma$, one on each boundary component.
We write
\[
\bA_\gamma' := \bA_\gamma \setminus \{p', p''\}.
\]

Let $W_\gamma$ denote the web diagram on $\bA_\gamma'$ obtained from $W$ by cutting along $\gamma$.
By Proposition~\ref{prop-bigon}(a), the web $W_\gamma$ is a non-elliptic
braided web on $\bA_\gamma'$.

The orientation of $\gamma$ induces orientations on the two boundary components of $\bA_\gamma'$.
There is a unique identification of $\bA_\gamma'$ with $\bA'$ such that the induced orientations
of the boundary components of $\bA_\gamma'$ agree with the orientation of the curve $C$
in Figure~\ref{fig:curve-C}.

We now introduce the following coordinate for $W$.

\begin{definition}\label{def-closed-torus}
Define
\[
\kappa \colon B_{\fS_1} \longrightarrow \BN^2 \times \BZ^2, \qquad
W \longmapsto (n_1(W), n_2(W), t_1(W), t_2(W))
:= {\bf t}(W_\gamma),
\]
where ${\bf t}$ is defined in Definition~\ref{def-twist-number}.
\end{definition}

The following shows the well-definedness of 
$\kappa$ in Definition~\ref{def-closed-torus}. 

\begin{lemma}
The map $\kappa$ defined in Definition~\ref{def-closed-torus} is well defined; that is,
it is independent of the choice of the general position of $W$ with respect to $\gamma$.
\end{lemma}

\begin{proof}
It suffices to prove the following statement.
Let $\beta$ be another oriented closed curve in $\fS_1$ parallel to $\gamma$, and fix a point $c \in \beta$.
Then
\[
{\bf t}(W_\gamma) = {\bf t}(W_\beta),
\]
where $W \in B_{\fS_1}$ is in general position with respect to both $\gamma$ and $\beta$.

By Proposition~\ref{prop-bigon}(b), we have
\[
n_1(W_\gamma) = n_1(W_\beta),
\qquad
n_2(W_\gamma) = n_2(W_\beta).
\]

The curves $\gamma$ and $\beta$ bound two annuli in $\fS_1$, which we denote by $\bA_1$ and $\bA_2$.
Set
\[
\bA_1' := \bA_1 \setminus \{p,c\},
\qquad
\bA_2' := \bA_2 \setminus \{p,c\}.
\]
By Proposition~\ref{prop-bigon}(a), the web
$W \cap \bA_1'$ (resp.\ $W \cap \bA_2'$) is a non-elliptic
braided web on $\bA_1'$ (resp.\ $\bA_2'$).

It follows immediately that
\[
t_j(W_\gamma)
=
t_j(W \cap \bA_1') + t_j(W \cap \bA_2')
=
t_j(W_\beta),
\qquad j=1,2.
\]
This proves the lemma.
\end{proof}

The following theorem is our second main result. It shows that the coordinates of non-elliptic web diagrams in the closed torus $\fS_1$ parametrize these web diagrams.

\begin{theorem}\label{thm-torus}
    Let $\gamma$ be an oriented closed curve in $\fS_1$, and fix a point $p \in \gamma$.
    Then the map
    \[
        \kappa \colon B_{\fS_1} \longrightarrow \BN^2 \times \BZ^2
    \]
    is injective. Moreover,
    \begin{align}\label{eq-im-ka-torus}
        \im \kappa
        =
        \bigl\{
        (n_1,n_2,t_1,t_2) \in \BN^2 \times \BZ^2
        \;\big|\;
        t_i \ge 0 \text{ if } n_i = 0,\ i=1,2
        \bigr\}.
    \end{align}
\end{theorem}

\begin{proof}
    We first prove the injectivity of $\kappa$.
    Suppose that $\kappa(W) = \kappa(D)$ for any two $W,D\in B_{\fS_1}$.
    Then $\bt(W_\gamma) = \bt(D_\gamma)$.

    Let
    \[
        {\bf s}(W_\gamma)
        =
        (s_1,\ldots,s_n,s_{n+1},\ldots,s_{2n}) \in \{-,+\}^{2n},
    \]
    where ${\bf s}$ is the signature defined in~\eqref{def-signature} and
    $n = i(\gamma,W)$.
    Since $W_\gamma$ is obtained from $W$ by cutting along $\gamma$, we have
    $s_i = s_{n+i}$ for $1 \le i \le n$.
    By applying a sequence of moves in Figure~\ref{fig:signature-move}, we obtain
    a braided web $W_1$ on $\bA_\gamma'$ such that
    ${\bf s}(W_1) = {\bf s}(D_\gamma)$.

    Let $W_2$ be obtained from $W_1$ by removing all $4$-gons as in
    Figure~\ref{fig:H-annulus}.
    Then
    \[
        {\bf t}(W_2)
        =
        {\bf t}(W_\gamma)
        =
        {\bf t}(D_\gamma),
        \qquad
        {\bf s}(W_2) = {\bf s}(D_\gamma).
    \]
    Lemma~\ref{lem-minimal-non}(a) and
    Theorem~\ref{prop-twist} therefore imply that $W_2 = D_\gamma$.

    When gluing $\bA_\gamma$ along $\gamma$ to recover $\fS_1$, we simultaneously
    glue $W_1$ (resp.\ $W_2$) to an almost non--elliptic web diagram
    $W_1'$ (resp.\ $W_2'$) in $\fS_1$.
    Since $W_2 = D_\gamma$, we have $W_2' = D$.
    Relations~\eqref{4-gon} and~\eqref{4-gon-re} imply that
    \[
        [W] = [W_1'] = [W_2'] = [D] \in \cS_C(\fS_1),
    \]
    where $\cS_C(\fS_1)$ is the graded algebra defined in~\eqref{def-graded-algebra}.
    By Lemma~\ref{lem-graded}, it follows that $W = D$.

    We now prove~\eqref{eq-im-ka-torus}.
    Theorem~\ref{prop-twist} shows that
    \[
        \im \kappa
        \subset
        \bigl\{
        (n_1,n_2,t_1,t_2) \in \BN^2 \times \BZ^2
        \;\big|\;
        t_i \ge 0 \text{ if } n_i = 0,\ i=1,2
        \bigr\}.
    \]

    Conversely, let $(n_1,n_2,t_1,t_2) \in \BN^2 \times \BZ^2$ satisfy
    $t_i \ge 0$ whenever $n_i = 0$ for $i=1,2$.
    Choose a signature
    ${\bf a} = (a_1,\ldots,a_n,a_{n+1},\ldots,a_{2n}) \in \{-,+\}^{2n}$
    with $a_i = a_{n+i}$ for $1 \le i \le n$.
    Theorem~\ref{prop-twist} guarantees the existence of a non-elliptic
    braided web $E$ on $\bA_\gamma'$ such that
    \[
        {\bf t}(E) = (n_1,n_2,t_1,t_2)\quad
        \text{and} \quad
        {\bf s}(E)={\bf a}.
    \]
    Gluing $\bA_\gamma$ along $\gamma$ yields an almost non-elliptic web diagram
    $E'$ in $\fS_1$.
    Let $E''$ be the non-elliptic web diagram obtained from $E'$ by removing all
    $4$-gons; this procedure is unique and satisfies $[E''] = [E']$.
    Then
    \[
        \kappa(E'')
        =
        {\bf t}(E''_\gamma)
        =
        {\bf t}(E)
        =
        (n_1,n_2,t_1,t_2),
    \]
    completing the proof.
\end{proof}

\begin{figure}
    \centering
    \includegraphics[width=0.7\linewidth]{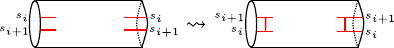}
    \caption{A move to swap signatures, where $s_i\neq s_{i+1}$.}
    \label{fig:signature-move}
\end{figure}

\bibliography{ref.bib}

\end{document}

%% file: 2-gon-un.tex
\raisebox{-.05in}{
	\begin{tikzpicture}
		\tikzset{->-/.style=
			
			{decoration={markings,mark=at position #1 with
					
					{\arrow{latex}}},postaction={decorate}}}
        \filldraw[draw=white,fill=gray!20] (0,0.1) rectangle (1.2, 0.7);
		\draw [line width =0.8pt](0,0.4)--(0.3,0.4);
        \draw [line width =0.8pt](0.9,0.4)--(1.2,0.4);
		\draw [line width =0.8pt] (0.6,0.4) ellipse (0.3cm and 0.2cm);
	\end{tikzpicture}
}

%% file: 4-gon-un.tex
\raisebox{-.05in}{
	\begin{tikzpicture}
		\tikzset{->-/.style=
			
			{decoration={markings,mark=at position #1 with
					
					{\arrow{latex}}},postaction={decorate}}}
        \filldraw[draw=white,fill=gray!20] (0.2,0.5) rectangle (1, 1.1);
		\draw [line width =0.8pt](0.4,0.6)--(0.2,0.6);
		\draw [line width =0.8pt](0.2,1)--(0.4,1);
		\draw [line width =0.8pt](1,0.6)--(0.8,0.6);
        \draw [line width =0.8pt](0.8,1)--(1,1);
        \draw[line width =0.8pt] (0.4,0.6)rectangle (0.8, 1);
	\end{tikzpicture}
}

%% file: un-trivial-knot.tex
\raisebox{-.10in}{
	\begin{tikzpicture}
		\tikzset{->-/.style=
			
			{decoration={markings,mark=at position #1 with
					
					{\arrow{latex}}},postaction={decorate}}}
        \filldraw[draw=white,fill=gray!20] (0,0) rectangle (0.8, 0.8);
        \draw[line width =0.8pt] (0.4,0.4) circle (0.3cm);
	\end{tikzpicture}
}

%% file: positive-crossing.tex
\raisebox{-.20in}{
	
	\begin{tikzpicture}
		\tikzset{->-/.style=
			
			{decoration={markings,mark=at position #1 with
					
					{\arrow{latex}}},postaction={decorate}}}
		\filldraw[draw=white,fill=gray!20] (-0,-0.2) rectangle (1, 1.2);
		\draw [line width =0.8pt,decoration={markings, mark=at position 0.5 with {\arrow{>}}},postaction={decorate}](0.6,0.6)--(1,1);
		\draw [line width =0.8pt,decoration={markings, mark=at position 0.5 with {\arrow{>}}},postaction={decorate}](0.6,0.4)--(1,0);
		\draw[line width =0.8pt] (0,0)--(0.4,0.4);
		\draw[line width =0.8pt] (0,1)--(0.4,0.6);
		\draw[line width =0.8pt] (0.4,0.6)--(0.6,0.4);
	\end{tikzpicture}
}

%% file: H-re.tex
\raisebox{-.20in}{
	\begin{tikzpicture}
		\tikzset{->-/.style=
			
			{decoration={markings,mark=at position #1 with
					
					{\arrow{latex}}},postaction={decorate}}}
		\filldraw[draw=white,fill=gray!20] (0,-0.2) rectangle (1.2, 1.2);
		\draw [line width =0.8pt,decoration={markings, mark=at position 0.7 with {\arrow{>}}},postaction={decorate}](0,0)--(0.4,0.5);
		\draw [line width =0.8pt,decoration={markings, mark=at position 0.7 with {\arrow{>}}},postaction={decorate}](0,1)--(0.4,0.5);
		\draw[line width =0.8pt] (0.4,0.5)--(0.8,0.5);
		\draw [line width =0.8pt,decoration={markings, mark=at position 0.6 with {\arrow{>}}},postaction={decorate}](0.8,0.5)--(1.2,0);
        \draw [line width =0.8pt,decoration={markings, mark=at position 0.6 with {\arrow{>}}},postaction={decorate}](0.8,0.5)--(1.2,1);
	\end{tikzpicture}
}

%% file: smooth-re.tex
\raisebox{-.20in}{
	
	\begin{tikzpicture}
		\tikzset{->-/.style=
			
			{decoration={markings,mark=at position #1 with
					
					{\arrow{latex}}},postaction={decorate}}}
		\filldraw[draw=white,fill=gray!20] (-0,-0.2) rectangle (1, 1.2);
		\draw [line width =0.8pt,decoration={markings, mark=at position 0.5 with {\arrow{>}}},postaction={decorate}](0,0.8)--(1,0.8);
		\draw [line width =0.8pt,decoration={markings, mark=at position 0.5 with {\arrow{>}}},postaction={decorate}](0,0.2)--(1,0.2);
	\end{tikzpicture}
}

%% file: negative-crossing.tex
\raisebox{-.20in}{
	\begin{tikzpicture}
		\tikzset{->-/.style=
			
			{decoration={markings,mark=at position #1 with
					
					{\arrow{latex}}},postaction={decorate}}}
		\filldraw[draw=white,fill=gray!20] (-0,-0.2) rectangle (1, 1.2);
		\draw [line width =1pt,decoration={markings, mark=at position 0.5 with {\arrow{>}}},postaction={decorate}](0.6,0.6)--(1,1);
		\draw [line width =0.8pt,decoration={markings, mark=at position 0.5 with {\arrow{>}}},postaction={decorate}](0.6,0.4)--(1,0);
		\draw[line width =0.8pt] (0,0)--(0.4,0.4);
		\draw[line width =0.8pt] (0,1)--(0.4,0.6);
		\draw[line width =0.8pt] (0.6,0.6)--(0.4,0.4);
	\end{tikzpicture}
}

%% file: 4-gon.tex
\raisebox{-.15in}{
	\begin{tikzpicture}
		\tikzset{->-/.style=
			
			{decoration={markings,mark=at position #1 with
					
					{\arrow{latex}}},postaction={decorate}}}
		\filldraw[draw=white,fill=gray!20] (0,0.2) rectangle (1.2, 1.2);
		\draw [line width =0.8pt,decoration={markings, mark=at position 0.7 with {\arrow{>}}},postaction={decorate}](0.4,0.4)--(0,0.4);
		\draw [line width =0.8pt,decoration={markings, mark=at position 0.7 with {\arrow{>}}},postaction={decorate}](0,1)--(0.4,1);
		\draw [line width =0.8pt,decoration={markings, mark=at position 0.6 with {\arrow{>}}},postaction={decorate}](1.2,0.4)--(0.8,0.4);
        \draw [line width =0.8pt,decoration={markings, mark=at position 0.6 with {\arrow{>}}},postaction={decorate}](0.8,1)--(1.2,1);
        \draw[line width =0.8pt] (0.4,0.4)rectangle (0.8, 1);
	\end{tikzpicture}
}

%% file: smooth2.tex
\raisebox{-.15in}{
	\begin{tikzpicture}
		\tikzset{->-/.style=
			
			{decoration={markings,mark=at position #1 with
					
					{\arrow{latex}}},postaction={decorate}}}
		\filldraw[draw=white,fill=gray!20] (0,0.2) rectangle (1.2, 1.2);
		\draw [line width =0.8pt,decoration={markings, mark=at position 0.5 with {\arrow{>}}},postaction={decorate}](1.2,0.4)--(0,0.4);
		\draw [line width =0.8pt,decoration={markings, mark=at position 0.5 with {\arrow{>}}},postaction={decorate}](0,1)--(1.2,1);
	\end{tikzpicture}
}

%% file: smooth1.tex
\raisebox{-.15in}{
	\begin{tikzpicture}
		\tikzset{->-/.style=
			
			{decoration={markings,mark=at position #1 with
					
					{\arrow{latex}}},postaction={decorate}}}
		\filldraw[draw=white,fill=gray!20] (0,0.2) rectangle (1.2, 1.2);
         \draw[line width =0.8pt,decoration={markings, mark=at position 0.7 with {\arrow{<}}},postaction={decorate}]
  plot[smooth] coordinates {(0,0.4) (0.5,0.7) (0,1)};
  \draw[line width =0.8pt,decoration={markings, mark=at position 0.7 with {\arrow{>}}},postaction={decorate}]
  plot[smooth] coordinates {(1.2,0.4) (0.7,0.7) (1.2,1)};
	\end{tikzpicture}
}

%% file: 4-gon-rev.tex
\raisebox{-.15in}{
	\begin{tikzpicture}
		\tikzset{->-/.style=
			
			{decoration={markings,mark=at position #1 with
					
					{\arrow{latex}}},postaction={decorate}}}
		\filldraw[draw=white,fill=gray!20] (0,0.2) rectangle (1.2, 1.2);
		\draw [line width =0.8pt,decoration={markings, mark=at position 0.7 with {\arrow{<}}},postaction={decorate}](0.4,0.4)--(0,0.4);
		\draw [line width =0.8pt,decoration={markings, mark=at position 0.7 with {\arrow{<}}},postaction={decorate}](0,1)--(0.4,1);
		\draw [line width =0.8pt,decoration={markings, mark=at position 0.6 with {\arrow{<}}},postaction={decorate}](1.2,0.4)--(0.8,0.4);
        \draw [line width =0.8pt,decoration={markings, mark=at position 0.6 with {\arrow{<}}},postaction={decorate}](0.8,1)--(1.2,1);
        \draw[line width =0.8pt] (0.4,0.4)rectangle (0.8, 1);
	\end{tikzpicture}
}

%% file: smooth2-re.tex
\raisebox{-.15in}{
	\begin{tikzpicture}
		\tikzset{->-/.style=
			
			{decoration={markings,mark=at position #1 with
					
					{\arrow{latex}}},postaction={decorate}}}
		\filldraw[draw=white,fill=gray!20] (0,0.2) rectangle (1.2, 1.2);
		\draw [line width =0.8pt,decoration={markings, mark=at position 0.5 with {\arrow{<}}},postaction={decorate}](1.2,0.4)--(0,0.4);
		\draw [line width =0.8pt,decoration={markings, mark=at position 0.5 with {\arrow{<}}},postaction={decorate}](0,1)--(1.2,1);
	\end{tikzpicture}
}

%% file: smooth1-re.tex
\raisebox{-.15in}{
	\begin{tikzpicture}
		\tikzset{->-/.style=
			
			{decoration={markings,mark=at position #1 with
					
					{\arrow{latex}}},postaction={decorate}}}
		\filldraw[draw=white,fill=gray!20] (0,0.2) rectangle (1.2, 1.2);
         \draw[line width =0.8pt,decoration={markings, mark=at position 0.7 with {\arrow{>}}},postaction={decorate}]
  plot[smooth] coordinates {(0,0.4) (0.5,0.7) (0,1)};
  \draw[line width =0.8pt,decoration={markings, mark=at position 0.7 with {\arrow{<}}},postaction={decorate}]
  plot[smooth] coordinates {(1.2,0.4) (0.7,0.7) (1.2,1)};
	\end{tikzpicture}
}

%% file: 2-gon.tex
\raisebox{-.10in}{
	\begin{tikzpicture}
		\tikzset{->-/.style=
			
			{decoration={markings,mark=at position #1 with
					
					{\arrow{latex}}},postaction={decorate}}}
		\filldraw[draw=white,fill=gray!20] (0,0) rectangle (1.2, 0.8);
		\draw [line width =0.8pt,decoration={markings, mark=at position 0.75 with {\arrow{>}}},postaction={decorate}](0,0.4)--(0.3,0.4);
        \draw [line width =0.8pt,decoration={markings, mark=at position 0.75 with {\arrow{>}}},postaction={decorate}](0.9,0.4)--(1.2,0.4);
		\draw [line width =0.8pt] (0.6,0.4) ellipse (0.3cm and 0.2cm);
	\end{tikzpicture}
}

%% file: straight-line.tex
\raisebox{-.10in}{
	\begin{tikzpicture}
		\tikzset{->-/.style=
			
			{decoration={markings,mark=at position #1 with
					
					{\arrow{latex}}},postaction={decorate}}}
		\filldraw[draw=white,fill=gray!20] (0,0) rectangle (1.2, 0.8);
        \draw [line width =0.8pt,decoration={markings, mark=at position 0.75 with {\arrow{>}}},postaction={decorate}](0,0.4)--(1.2,0.4);
	\end{tikzpicture}
}

%% file: trivial-knot1.tex
\begin{tikzpicture}
\filldraw[draw=white,fill=gray!20] (0,0) rectangle (1, 1);
\draw[<-, line width=0.8pt]
  (0.85,0.5) arc (0:360:0.35cm);
\end{tikzpicture}

%% file: trivial-knot2.tex
\raisebox{-.15in}{
\begin{tikzpicture}
\filldraw[draw=white,fill=gray!20] (0,0) rectangle (1, 1);
\draw[->, line width=0.8pt]
  (0.85,0.5) arc (0:360:0.35cm);
\end{tikzpicture}
}

%% file: trivial-arc.tex
\raisebox{-.05in}{
	\begin{tikzpicture}
		\tikzset{->-/.style=
			
			{decoration={markings,mark=at position #1 with
					
					{\arrow{latex}}},postaction={decorate}}}
        \filldraw[draw=white,fill=gray!20] (0,0) rectangle (1, 0.6);
        \draw[line width =0.8pt] (0.8,0) arc (0:180:0.3cm);
		\draw [line width =1.5pt](0,0)--(1,0);
	\end{tikzpicture}
}

%% file: trivial-Y.tex
\raisebox{-.05in}{
	\begin{tikzpicture}
		\tikzset{->-/.style=
			
			{decoration={markings,mark=at position #1 with
					
					{\arrow{latex}}},postaction={decorate}}}
        \filldraw[draw=white,fill=gray!20] (0,0) rectangle (1, 0.6);
        \draw[line width =0.8pt] (0.5,0.3)--(0.5,0.6);
        \draw[line width =0.8pt] (0.5,0.3)--(0.2,0);
        \draw[line width =0.8pt] (0.5,0.3)--(0.8,0);
		\draw [line width =1.5pt](0,0)--(1,0);
	\end{tikzpicture}
}

%% file: H-boundary.tex
\raisebox{-.08in}{
	\begin{tikzpicture}
		\tikzset{->-/.style=
			
			{decoration={markings,mark=at position #1 with
					
					{\arrow{latex}}},postaction={decorate}}}
        \filldraw[draw=white,fill=gray!20] (0,0) rectangle (0.8, 0.6);
        \draw[line width =0.8pt] (0.2,0)--(0.2,0.6);
        \draw[line width =0.8pt] (0.6,0)--(0.6,0.6);
        \draw[line width =0.8pt] (0.2,0.3)--(0.6,0.3);
		\draw [line width =1.5pt](0,0)--(0.8,0);
	\end{tikzpicture}
}

%% file: elliptic-circle.tex
\raisebox{-.10in}{
	\begin{tikzpicture}
		\tikzset{->-/.style=
			
			{decoration={markings,mark=at position #1 with
					
					{\arrow{latex}}},postaction={decorate}}}
        \filldraw[draw=white,fill=gray!20] (0,0) rectangle (1, 1);
        \draw[line width =0.8pt] (0.5,0.5) circle (0.3cm);
        \draw[fill=black,line width =0.8pt] (0.5,0.2) circle (0.05cm);
	\end{tikzpicture}
},\qquad
\raisebox{-.10in}{
	\begin{tikzpicture}
		\tikzset{->-/.style=
			
			{decoration={markings,mark=at position #1 with
					
					{\arrow{latex}}},postaction={decorate}}}
        \filldraw[draw=white,fill=gray!20] (0,0) rectangle (1, 1);
        \draw[line width =0.8pt] (0.5,0.6)--(0.5,1);
        \draw[line width =0.8pt] (0.5,0.4) circle (0.2cm);
        \draw[fill=black,line width =0.8pt] (0.5,0.2) circle (0.05cm);
	\end{tikzpicture}
}.